\documentclass[12pt, a4paper, reqno]{amsart}
\usepackage{amsmath}
\usepackage{amsfonts}
\usepackage{amssymb}
\usepackage{color}
\usepackage{comment,graphicx}
\usepackage[T1]{fontenc}
\usepackage{url}
\usepackage{ae}

\def\phi{\varphi}
\def\rho{\varrho}
\def\epsilon{\varepsilon}
\numberwithin{equation}{section}
\theoremstyle{plain}
\newtheorem{theorem}[equation]{Theorem}
\newtheorem{lemma}[equation]{Lemma}

\theoremstyle{definition}
\newtheorem{definition}[equation]{Definition}

\theoremstyle{remark}
\newtheorem{remark}[equation]{Remark}

\renewcommand{\le}{\leqslant}
\renewcommand{\ge}{\geqslant}
\renewcommand{\leq}{\leqslant}
\renewcommand{\geq}{\geqslant}

\begin{document}
	\title[Herz-type Sobolev spaces on domains]{Multiplication properties and  the  Rellich-Kondrachov  theorem in Herz-type Sobolev spaces}
	\author{Douadi Drihem}
	\maketitle
	\date{\today }
	
	\begin{abstract}
		In this paper, we prove that Herz-type Sobolev spaces form a Banach algebra and establish a Rellich--Kondrachov compactness theorem for these spaces. These results extend the corresponding classical theory and further demonstrate that Herz-type Sobolev spaces provide a natural generalization of the classical Sobolev spaces.

		
	\end{abstract}
\section{Introduction}

Sobolev spaces are among the most fundamental function spaces in modern analysis, with numerous applications in harmonic analysis and the theory of partial differential equations. We refer the reader to the monographs \cite{AdamsFournier03}, \cite{Burenkov}, \cite{Ma85}, and \cite{Tartar} for comprehensive treatments, historical background, and further references on Sobolev spaces.

Herz spaces, introduced by Herz in \cite{Herz68}, also occupy an important place in harmonic analysis, particularly in the study of singular integral operators, Fourier multipliers, and related topics. Their significance arises not only from their rich mathematical structure but also from their wide range of applications in analysis. We refer the reader to \cite{BS85}, \cite{Dr-Herz-Heat}, \cite{Fi-We08}, \cite{Rag09}, and \cite{Tu11} for some of these applications. More recent developments and further advances in the theory of Herz spaces can be found in \cite{RS}, \cite{ZYZ}, and the monograph \cite{LYH}.

Motivated by the theories of Sobolev and Herz spaces, the author in \cite{Dr-Sobolev,Dr-Herz-Sobolev2} introduced the class of Herz-type Sobolev spaces on domains and established several of their fundamental properties, including Sobolev-type embedding theorems. Closely related function spaces on $\mathbb{R}^{n}$ were introduced earlier by Lu and Yang in \cite{LuYang97}, who also demonstrated their applicability to the study of partial differential equations.

In this paper, we  show that Herz-type Sobolev spaces form a Banach algebra and establish a Rellich--Kondrachov compactness theorem for these spaces.

The paper is organized as follows. In Section 2, we present the necessary preliminaries, fix the notation, and recall some basic facts concerning Herz spaces, including interpolation inequalities and density results.

In Section 3, we recall the fundamental properties of Herz-type Sobolev spaces, highlighting their analogy with the classical Sobolev spaces. We also present Sobolev embedding theorems for these spaces. Finally, in Section 4, we establish pointwise multiplication properties, proving that Herz-type Sobolev spaces form a Banach algebra under suitable assumptions, and we conclude with a Rellich--Kondrachov compactness theorem for these spaces.

\subsection{Notation and conventions}

As usual, $\mathbb{R}^{n}$ denotes the $n$-dimensional Euclidean space, $%
\mathbb{N}$ the set of all natural numbers, and $\mathbb{N}_{0}=\mathbb{N}%
\cup \{0\}$. The symbol $\mathbb{Z}$ stands for the set of all integers. 
For $x\in \mathbb{R}^{n}$ and $r>0$, we denote by $B(x,r)$ the open ball in $%
\mathbb{R}^{n}$ centered at $x$ with radius $r$. 
If $1\leq p<\infty$ and 
$\frac{1}{p}+\frac{1}{p^{\prime }}=1,$
then $p^{\prime }$ is called the conjugate exponent of $p$.

\medskip

We denote by $|\Omega|$ the $n$-dimensional Lebesgue measure of a measurable
set $\Omega \subseteq \mathbb{R}^{n}$. For any measurable set $\Omega
\subseteq \mathbb{R}^{n}$, the Lebesgue space $L^{p}(\Omega)$, $0<p\leq
\infty$, consists of all measurable functions $f$ such that 
\begin{equation*}
\|f\|_{L^{p}(\Omega)} =\left( \int_{\Omega}|f(x)|^{p}\,dx
\right)^{1/p}<\infty, \qquad 0<p<\infty,
\end{equation*}
and 
\begin{equation*}
\big\Vert f\big\Vert_{L^{\infty }(\Omega )}=\underset{x\in \Omega }{\text{%
		ess-sup}}\left\vert f(x)\right\vert <\infty .
\end{equation*}
If $\Omega=\mathbb{R}^{n}$, we simply write 
\begin{equation*}
\|f\|_{L^{p}(\mathbb{R}^{n})}=\|f\|_{p}. 
\end{equation*}

\medskip

If $G\subset \mathbb{R}^{n}$ is nonempty, we denote by $\overline{G}$ the
closure of $G$ in $\mathbb{R}^{n}$. Let $\Omega \subseteq \mathbb{R}^{n}$ be
open. We write $G\Subset \Omega$ if $\overline{G}\subset \Omega$ and $G$ is
compact in $\mathbb{R}^{n}$. For any nonnegative integer $m$, let $%
C^{m}(\Omega)$ denote the vector space of all functions $f$ such that $f$
and all their partial derivatives $D^{\beta}f$ of order $|\beta|\leq m$ are
continuous on $\Omega$. We set 
\begin{equation*}
C^{0}(\Omega)=C(\Omega), \qquad C^{\infty}(\Omega)=\bigcap_{m\geq 0}
C^{m}(\Omega).
\end{equation*}
Finally, we denote by $C_{c}(\Omega)$ the space of all continuous functions
on $\Omega$ with compact support in $\Omega$.

\section{Herz spaces}

In this section, we present some fundamental properties of Herz spaces. Let $%
k\in \mathbb{Z}$. For convenience, we set 
\begin{equation*}
B_{k}=B(0,2^{k}), \qquad R_{k}=B_{k}\setminus B_{k-1}.
\end{equation*}
We denote by $\chi_{k}$ the characteristic function of the set $R_{k}$.

\begin{definition}
	\label{def-inh-Herz} Let $\alpha \in \mathbb{R}$ and $1\leq p,q\leq \infty$.
	
	\medskip
	
	\noindent \textnormal{(i)} The homogeneous Herz space $\dot{K}%
	_{p}^{\alpha,q}(\mathbb{R}^{n})$ is defined as the set of all functions $%
	f\in L_{\mathrm{loc}}^{p}(\mathbb{R}^{n}\setminus\{0\})$ such that 
	\begin{equation*}
	\|f\|_{\dot{K}_{p}^{\alpha,q}(\mathbb{R}^{n})} = \left(
	\sum_{k=-\infty}^{\infty} 2^{k\alpha q} \|f\chi_{k}\|_{p}^{q} \right)^{1/q}
	<\infty,
	\end{equation*}
	with the usual modifications when $p=\infty$ and/or $q=\infty$.
	
	\medskip
	
	\noindent \textnormal{(ii)} The nonhomogeneous Herz space $K_{p}^{\alpha,q}(%
	\mathbb{R}^{n})$ is defined as the set of all functions $f\in L_{\mathrm{loc}%
	}^{p}(\mathbb{R}^{n})$ such that 
	\begin{equation*}
	\|f\|_{K_{p}^{\alpha,q}(\mathbb{R}^{n})} = \|f\chi_{B_{0}}\|_{p} + \left(
	\sum_{k=1}^{\infty} 2^{k\alpha q} \|f\chi_{k}\|_{p}^{q} \right)^{1/q}
	<\infty,
	\end{equation*}
	with the usual modifications when $p=\infty$ and/or $q=\infty$.
\end{definition}

\begin{remark}
	The spaces $K_{p}^{\alpha,q}(\mathbb{R}^{n})$ and $\dot{K}_{p}^{\alpha,q}(%
	\mathbb{R}^{n})$ are Banach spaces. If $\alpha=0$ and $1\leq p=q\leq \infty$%
	, then both $\dot{K}_{p}^{0,p}(\mathbb{R}^{n})$ and $K_{p}^{0,p}(\mathbb{R}%
	^{n})$ coincide with the Lebesgue space $L^{p}(\mathbb{R}^{n})$. Moreover, 
	\begin{equation*}
	\dot{K}_{p}^{\alpha,p}(\mathbb{R}^{n}) = L^{p}(\mathbb{R}^{n},|\cdot|^{%
		\alpha p}),
	\end{equation*}
	that is, the weighted Lebesgue space endowed with the norm 
	\begin{equation*}
	\|f\|_{L^{p}(\mathbb{R}^{n},|\cdot|^{\alpha p})} = \left( \int_{\mathbb{R}%
		^{n}} |f(x)|^{p}|x|^{\alpha p}\,dx \right)^{1/p}.
	\end{equation*}
	
	Let $0<p,q\leq \infty$ and $\alpha>0$. Then 
	\begin{equation*}
	\dot{K}_{p}^{\alpha,q}(\mathbb{R}^{n})\cap L^{p}(\mathbb{R}^{n}) =
	K_{p}^{\alpha,q}(\mathbb{R}^{n}),
	\end{equation*}
	in the sense of equivalent quasi-norms. Furthermore, 
	\begin{equation*}
	\|f\|_{K_{p}^{\alpha,q}(\mathbb{R}^{n})} \approx \left( \sum_{k=0}^{\infty}
	2^{k\alpha q} \|f\chi_{\widehat{R}_{k}}\|_{p}^{q} \right)^{1/q},
	\end{equation*}
	where $\widehat{R}_{k}=R_{k}$ for $k\in \mathbb{N}$ and $\widehat{R}%
	_{0}=B_{0}$.
	
	Let $0<p\leq \infty$, $0<q_{1}\leq q_{2}\leq \infty$, and $\alpha\in\mathbb{R}$.
	Then the following continuous embeddings hold:
	\begin{equation}
	\dot{K}_{p}^{\alpha,q_{1}}(\mathbb{R}^{n})
	\hookrightarrow
	\dot{K}_{p}^{\alpha,q_{2}}(\mathbb{R}^{n}),
	\qquad
	K_{p}^{\alpha,q_{1}}(\mathbb{R}^{n})
	\hookrightarrow
	K_{p}^{\alpha,q_{2}}(\mathbb{R}^{n}).
	\label{herz-emb}
	\end{equation}
	
	Moreover, if $\alpha_{2}\leq \alpha_{1}$ and $0<q\leq \infty$, then
	\begin{equation}
	K_{p}^{\alpha_{1},q}(\mathbb{R}^{n})
	\hookrightarrow
	K_{p}^{\alpha_{2},q}(\mathbb{R}^{n}).
	\label{herz-emb1}
	\end{equation}
\end{remark}

A detailed discussion of the properties of these spaces may be found in the
monograph \cite{LYH08}, the papers \cite{LuYang1.95}, \cite{LuYang2.95}, 
\cite{Hernandez1998}, and the references therein.

If $\Omega \subset \mathbb{R}^{n}$ is open and $f:\Omega \rightarrow \mathbb{%
	R}$ is measurable, then we write $f\in \dot{K}_{p}^{\alpha,q}(\Omega)$
whenever $f\chi_{\Omega}\in \dot{K}_{p}^{\alpha,q}(\mathbb{R}^{n})$, and we
define 
\begin{equation*}
\|f\|_{\dot{K}_{p}^{\alpha,q}(\Omega)} = \|f\chi_{\Omega}\|_{\dot{K}%
	_{p}^{\alpha,q}(\mathbb{R}^{n})}.
\end{equation*}
Similarly, we define 
\begin{equation*}
\|f\|_{K_{p}^{\alpha,q}(\Omega)} = \|f\chi_{\Omega}\|_{K_{p}^{\alpha,q}(%
	\mathbb{R}^{n})}.
\end{equation*}

Let $V_{\alpha,p,q}$ denote the set of all $(\alpha,p,q)\in \mathbb{R}\times
[1,\infty]^{2}$ such that:

\begin{itemize}
	\item $\alpha<n-\dfrac{n}{p}$, $1\leq p\leq \infty$, and $1\leq q\leq \infty$%
	;
	
	\item $\alpha=n-\dfrac{n}{p}$, $1\leq p\leq \infty$, and $q=1$.
\end{itemize}

The next lemma gives a necessary and sufficient condition on the parameters $%
\alpha$, $p$, and $q$ ensuring that 
\begin{equation*}
\langle T_{f},\varphi\rangle = \int_{\Omega}f(x)\varphi(x)\,dx, \qquad
\varphi\in \mathcal{D}(\Omega),
\end{equation*}
defines a regular distribution $T_{f}\in \mathcal{D}^{\prime }(\Omega)$ for
every $f\in \dot{K}_{p}^{\alpha,q}(\Omega)$.

\begin{lemma}
	\label{L1loc} Let $\Omega \subset \mathbb{R}^{n}$ be open with $0\in \Omega$%
	, and let $1\leq p,q\leq \infty$ and $\alpha\in \mathbb{R}$. Then 
	\begin{equation*}
	\dot{K}_{p}^{\alpha,q}(\Omega) \hookrightarrow L_{\mathrm{loc}}^{1}(\Omega)
	\end{equation*}
	if and only if $(\alpha,p,q)\in V_{\alpha,p,q}$.
\end{lemma}

\begin{remark}
	\label{L1loc1}
	For the proof of Lemma \ref{L1loc}, see \cite{Dr-Sobolev}. Let $\Omega
	\subset \mathbb{R}^{n}$ be open, $1\leq p,q\leq \infty$, and $\alpha\in 
	\mathbb{R}$. Then clearly 
	\begin{equation*}
	K_{p}^{\alpha,q}(\Omega) \hookrightarrow L_{\mathrm{loc}}^{1}(\Omega).
	\end{equation*}
\end{remark}

\begin{remark}
	In general, if $0\notin \Omega$, then the set $V_{\alpha,p,q}$ is not
	optimal. Consequently, Lemma \ref{L1loc} and Remark 	\ref{L1loc1} justify the definition of weak derivatives for functions belonging to Herz spaces.
\end{remark}

\begin{theorem}
	\label{dense} Let $\Omega \subset \mathbb{R}^{n}$ be an open set, $1\leq
	p<\infty$, and $1\leq q<\infty$.
	
	\medskip
	
	\noindent \textnormal{(i)} If $\alpha>-\dfrac{n}{p}$, then $C_{c}(\Omega)$
	is dense in $\dot{K}_{p}^{\alpha,q}(\Omega)$.
	
	\medskip
	
	\noindent \textnormal{(ii)} If $\alpha\in \mathbb{R}$, then $C_{c}(\Omega)$
	is dense in $K_{p}^{\alpha,q}(\Omega)$.
\end{theorem}

\begin{proof}
	The proof of \textnormal{(i)} is given in \cite{Dr-Sobolev}. Assertion 
	\textnormal{(ii)} can be proved similarly.
\end{proof}

We now present an interpolation inequality for Herz spaces.

\begin{lemma}
	\label{interpolation2} Let $\Omega \subset \mathbb{R}^{n}$ be an open set, $%
	0<p_{0},p_{1},q_{0},q_{1}\leq \infty$, and $\alpha_{0},\alpha_{1}\in \mathbb{%
		R}$. Define 
	\begin{equation*}
	\alpha = (1-\theta)\alpha_{0} + \theta\alpha_{1}, \qquad \frac{1}{p} = \frac{%
		1-\theta}{p_{0}} + \frac{\theta}{p_{1}}, \qquad \frac{1}{q} = \frac{1-\theta%
	}{q_{0}} + \frac{\theta}{q_{1}}.
	\end{equation*}
	Then the interpolation inequalities 
	\begin{equation}
	\|f\|_{\dot{K}_{p}^{\alpha,q}(\Omega)} \leq \|f\|_{\dot{K}%
		_{p_{0}}^{\alpha_{0},q_{0}}(\Omega)}^{1-\theta} \|f\|_{\dot{K}%
		_{p_{1}}^{\alpha_{1},q_{1}}(\Omega)}^{\theta}  \label{Interpolation}
	\end{equation}
	hold for all $f\in \dot{K}_{p_{0}}^{\alpha_{0},q_{0}}(\Omega) \cap \dot{K}%
	_{p_{1}}^{\alpha_{1},q_{1}}(\Omega)$, and 
	\begin{equation}
	\|f\|_{K_{p}^{\alpha,q}(\Omega)} \leq
	\|f\|_{K_{p_{0}}^{\alpha_{0},q_{0}}(\Omega)}^{1-\theta}
	\|f\|_{K_{p_{1}}^{\alpha_{1},q_{1}}(\Omega)}^{\theta}  \label{Interpolation1}
	\end{equation}
	hold for all $f\in K_{p_{0}}^{\alpha_{0},q_{0}}(\Omega) \cap
	K_{p_{1}}^{\alpha_{1},q_{1}}(\Omega)$.
\end{lemma}

\section{Herz-type Sobolev spaces}

In this section, we establish the basic properties of Herz-type Sobolev
spaces in analogy with the classical Sobolev spaces.

\begin{definition}
	Let $\Omega \subset \mathbb{R}^{n}$ be open, $\alpha \in \mathbb{R}$, $1\leq
	p,q\leq \infty$, and $m\in \mathbb{N}_{0}$.
	
	\medskip
	
	\noindent \textnormal{(i)} Let $(\alpha,p,q)\in V_{\alpha,p,q}$. The
	homogeneous Herz-type Sobolev space $\dot{K}_{p,m}^{\alpha,q}(\Omega)$ is
	defined as the set of all functions $f\in \dot{K}_{p}^{\alpha,q}(\Omega)$
	whose weak derivatives $D^{\beta}f$ belong to $\dot{K}_{p}^{\alpha,q}(%
	\Omega) $ for all multi-indices $\beta$ satisfying $|\beta|\leq m$.
	
	If $1\leq p,q<\infty$, we define the norm by 
	\begin{equation*}
	\|f\|_{\dot{K}_{p,m}^{\alpha,q}(\Omega)} = \left( \sum_{k=-\infty}^{\infty}
	2^{k\alpha q} \left( \sum_{|\beta|\leq m} \|(D^{\beta}f)\chi_{R_{k}\cap
		\Omega}\|_{p}^{p} \right)^{q/p} \right)^{1/q},
	\end{equation*}
	while for $q=\infty$ we set 
	\begin{equation*}
	\|f\|_{\dot{K}_{p,m}^{\alpha,\infty}(\Omega)} = \sup_{k\in \mathbb{Z}}
	2^{k\alpha} \left( \sum_{|\beta|\leq m} \|(D^{\beta}f)\chi_{R_{k}\cap
		\Omega}\|_{p}^{p} \right)^{1/p}.
	\end{equation*}
	
	\medskip
	
	\noindent \textnormal{(ii)} The nonhomogeneous Herz-type Sobolev space $%
	K_{p,m}^{\alpha,q}(\Omega)$ is defined as the set of all functions $f\in
	K_{p}^{\alpha,q}(\Omega)$ whose weak derivatives $D^{\beta}f$ belong to $%
	K_{p}^{\alpha,q}(\Omega)$ for all multi-indices $\beta$ satisfying $%
	|\beta|\leq m$.
	
	If $1\leq p,q<\infty$, we define the norm by 
	\begin{equation*}
	\|f\|_{K_{p,m}^{\alpha,q}(\Omega)} = \left( \sum_{k=0}^{\infty} 2^{k\alpha
		q} \left( \sum_{|\beta|\leq m} \|(D^{\beta}f)\chi_{\widehat{R}_{k}\cap
		\Omega}\|_{p}^{p} \right)^{q/p} \right)^{1/q},
	\end{equation*}
	while for $q=\infty$ we set 
	\begin{equation*}
	\|f\|_{K_{p,m}^{\alpha,\infty}(\Omega)} = \sup_{k\in \mathbb{N}_{0}}
	2^{k\alpha} \left( \sum_{|\beta|\leq m} \|(D^{\beta}f)\chi_{\widehat{R}%
		_{k}\cap \Omega}\|_{p}^{p} \right)^{1/p}.
	\end{equation*}
\end{definition}

\begin{remark}
	It is immediate that if $p=q$ and $\alpha=0$, then 
	\begin{equation*}
	\dot{K}_{p,m}^{0,p}(\Omega) = K_{p,m}^{0,p}(\Omega) = W_{p}^{m}(\Omega).
	\end{equation*}
\end{remark}

As in the classical Sobolev setting (see \cite{AdamsFournier03} and \cite%
{Burenkov}), we also have the following results.

\begin{theorem}
	Let $\Omega \subset \mathbb{R}^{n}$ be open, $\alpha\in \mathbb{R}$, $1\leq
	p,q\leq \infty$, and $m\in \mathbb{N}_{0}$.
	
	\medskip
	
	\noindent \textnormal{(i)} Let $(\alpha,p,q)\in V_{\alpha,p,q}$. Then the
	homogeneous Herz-type Sobolev space $\dot{K}_{p,m}^{\alpha,q}(\Omega)$ is a
	Banach space.
	
	\medskip
	
	\noindent \textnormal{(ii)} The nonhomogeneous Herz-type Sobolev space $%
	K_{p,m}^{\alpha,q}(\Omega)$ is a Banach space.
\end{theorem}

\begin{theorem}
	\label{density} Let $\Omega \subset \mathbb{R}^{n}$ be open, $\alpha\in 
	\mathbb{R}$, and $m\in \mathbb{N}_{0}$.
	
	\medskip
	
	\noindent \textnormal{(i)} Let $1\leq p,q<\infty$ and 
	\begin{equation*}
	-\frac{n}{p}<\alpha<n-\frac{n}{p},
	\end{equation*}
	or 
	\begin{equation*}
	1\leq p<\infty, \qquad \alpha=n-\frac{n}{p}, \qquad q=1.
	\end{equation*}
	Then  $C^{\infty}(\Omega)\cap \dot{K}_{p,m}^{\alpha,q}(\Omega)$ is dense in $\dot{K}_{p,m}^{\alpha,q}(\Omega)$.
	
	\medskip
	
	\noindent \textnormal{(ii)} Let $1\leq p,q<\infty$. Then 
	$C^{\infty}(\Omega)\cap K_{p,m}^{\alpha,q}(\Omega) $
	is dense in $K_{p,m}^{\alpha,q}(\Omega)$.
\end{theorem}
For the  proof see; \cite{Dr-Herz-Sobolev2}.

\subsection{Sobolev embeddings}

In this subsection, we recall several embedding results for Herz-type Sobolev spaces proved in \cite{Dr-Herz-Sobolev2}.

\begin{definition}
	Let $v\in \mathbb{R}^{n}\setminus \{0\}$ and, for each $x\neq 0$, let $%
	\angle (x,v)$ denote the angle between the position vector $x$ and $v$. Let $%
	\kappa$ satisfy $0<\kappa <\pi$. The set 
	\begin{equation*}
	C=\{x\in \mathbb{R}^{n}:x=0 \text{ or } 0<|x|\leq \delta,\ \angle (x,v)\leq
	\kappa /2\}
	\end{equation*}
	is called a finite cone with vertex at the origin, height $\delta$, axis
	direction $v$, and aperture angle $\kappa$.
\end{definition}

\begin{remark}
	Let $C$ be a finite cone with vertex at the origin. Observe that 
	\begin{equation*}
	x+C=\{x+y:y\in C\}
	\end{equation*}
	is a finite cone with vertex at $x$, having the same dimensions and axis
	direction as $C$, obtained by a parallel translation of $C$.
\end{remark}

We are now in a position to define domains satisfying the cone condition.

\begin{definition}
	\label{conecondition} Let $\Omega \subset \mathbb{R}^{n}$ be open. We say
	that $\Omega$ satisfies the cone condition if there exists a finite cone $C$
	such that, for every $x\in \Omega$, there exists a finite cone $C_{x}\subset
	\Omega$ with vertex at $x$ that is congruent to $C$.
\end{definition}

\begin{remark}
	In Definition \ref{conecondition}, the cone $C_{x}$ is not necessarily
	obtained from $C$ by parallel translation, but rather by a rigid motion.
\end{remark}
After these preparations, we are ready to state the first Sobolev embedding
theorem.

\begin{theorem}
	\label{embeddingsfirst} Let $\Omega \subset \mathbb{R}^{n}$ be a domain
	satisfying the cone condition, and let $m\in \mathbb{N}$. Assume that 
	$
	1<p<\infty,  1\leq r<\infty,
	$
	and 
	\begin{equation}
	\frac{n}{q} = \frac{n}{p}-m-\alpha _{1}+\alpha _{2} >0, \qquad -\frac{n}{q}%
	<\alpha _{1}\leq \alpha _{2}<n-\frac{n}{p}.  \label{sobolevassumption1}
	\end{equation}
	Assume furthermore that 
	$m>\alpha _{2}-\alpha _{1}.$
	Then 
	\begin{equation*}
	\dot{K}_{p,m}^{\alpha _{2},r}(\Omega ) \hookrightarrow \dot{K}_{q}^{\alpha
		_{1},r}(\Omega )
	\end{equation*}
	and 
	\begin{equation*}
	K_{p,m}^{\alpha _{2},r}(\Omega ) \hookrightarrow K_{q}^{\alpha
		_{1},r}(\Omega ).
	\end{equation*}
\end{theorem}

\begin{remark}
	Let $\Omega \subset \mathbb{R}^{n}$ be a domain satisfying the cone
	condition. Theorem \ref{embeddingsfirst} includes the classical Sobolev
	inequality as a special case. Moreover, by combining Theorem \ref%
	{embeddingsfirst} with the embeddings \eqref{herz-emb}, we obtain 
	\begin{equation*}
	W_{p}^{m}(\Omega ,|\cdot |^{\alpha _{2}p}) \hookrightarrow \dot{K}%
	_{q}^{\alpha _{1},p}(\Omega ) \hookrightarrow L^{q}(\Omega ,|\cdot |^{\alpha
		_{1}q}),
	\end{equation*}
	under the assumptions of Theorem \ref{embeddingsfirst} with $r=p$. In
	particular, 
	\begin{equation*}
	W_{p}^{m}(\Omega ) \hookrightarrow \dot{K}_{q}^{0,p}(\Omega )
	\hookrightarrow L^{q}(\Omega )
	\end{equation*}
	whenever 
	$
	1<p<\infty,  0<m<\frac{n}{p},
	$
	and 
	\begin{equation*}
	\frac{n}{q}=\frac{n}{p}-m.
	\end{equation*}
\end{remark}

In the next theorem, we consider the case $p=q$ in Theorem \ref%
{embeddingsfirst}. 

\begin{theorem}
	\label{embedingsq=p} Let the domain $\Omega \subset \mathbb{R}^{n}$ satisfy
	the cone condition, and let $m\in \mathbb{N}$. Assume that $1\le p<\infty$, $%
	1\le r<\infty$, and 
	\begin{equation*}
	-\frac{n}{p}<\alpha_{1}\le \alpha_{2}<n-\frac{n}{p}, \qquad m\ge
	\alpha_{2}-\alpha_{1}.
	\end{equation*}
	Then 
	\begin{equation*}
	\dot{K}_{p,m}^{\alpha_{2},r}(\Omega) \hookrightarrow \dot{K}%
	_{p}^{\alpha_{1},r}(\Omega)
	\end{equation*}
	and 
	\begin{equation*}
	K_{p,m}^{\alpha_{2},r}(\Omega) \hookrightarrow K_{p}^{\alpha_{1},r}(\Omega).
	\end{equation*}
\end{theorem}
Collecting the results of Theorems \ref{embedingsq=p} and \ref%
{embeddingsfirst}, together with the interpolation inequalities %
\eqref{Interpolation} and \eqref{Interpolation1}, we obtain the following
statement.

\begin{theorem}
	\label{embeddingsfirst copy(1)} Let $\Omega \subset \mathbb{R}^{n}$ be a
	domain satisfying the cone condition, and let $m\in \mathbb{N}_{0}$. Assume
	that $1\le r<\infty, 1<p<q<p^{\ast}<\infty, m>\alpha_{2}-\alpha_{1},$
	and 
	\begin{equation*}
	\frac{n}{p^{\ast}} = \frac{n}{p}-m+\alpha_{2}-\alpha_{1} >0,
	\end{equation*}
	with 
	\begin{equation*}
	-\frac{n}{q}<\alpha_{1}\le \alpha_{2}<n-\frac{n}{p}.
	\end{equation*}
	Then 
	\begin{equation*}
	\dot{K}_{p,m}^{\alpha_{2},r}(\Omega) \hookrightarrow \dot{K}%
	_{q}^{\alpha_{1},r}(\Omega)
	\end{equation*}
	and 
	\begin{equation}
	K_{p,m}^{\alpha_{2},r}(\Omega) \hookrightarrow K_{q}^{\alpha_{1},r}(\Omega).
	\label{emb9}
	\end{equation}
\end{theorem}
Next, we consider the case $q=\infty$.

\begin{theorem}
	\label{embeddingsq=infinity} Let $\Omega \subset \mathbb{R}^{n}$ be a domain
	satisfying the cone condition, and let $m\in \mathbb{N}$. Assume that 
	$1<p<\infty,  1\le r<\infty,0\le \alpha_{1}\le \alpha_{2}<n-\frac{n}{p}$
	and 
	\begin{equation*}
	m>\frac{n}{p}+\alpha_{2}-\alpha_{1}.
	\end{equation*}
	Then 
	\begin{equation}
	\dot{K}_{p,m}^{\alpha_{2},r}(\Omega) \hookrightarrow \dot{K}%
	_{\infty}^{\alpha_{1},v}(\Omega),  \label{q=infinity}
	\end{equation}
	and 
	\begin{equation*}
	K_{p,m}^{\alpha_{2},r}(\Omega) \hookrightarrow
	K_{\infty}^{\alpha_{1},v}(\Omega),
	\end{equation*}
	where 
	\begin{equation*}
	v= \left\{ 
	\begin{array}{ll}
	\infty, & \text{if } \alpha_{1}=0, \\[1mm] 
	r, & \text{if } \alpha_{1}>0.%
	\end{array}
	\right.
	\end{equation*}
\end{theorem}

We now study the  case 
\begin{equation*}
m\ge \frac{n}{p}+\alpha_{2}-\alpha_{1}
\end{equation*}
in Theorem \ref{embeddingsfirst copy(1)}.

\begin{theorem}
	\label{embedingsp<q} Let $\Omega \subset \mathbb{R}^{n}$ be a domain
	satisfying the cone condition, and let $m\in \mathbb{N}$. Assume that 
	$1<p<q<\infty, 1\le r<\infty,$
	$-\frac{n}{q}<\alpha_{1}\le \alpha_{2}<n-\frac{n}{p},$
	and 
	\begin{equation*}
	m\ge \frac{n}{p}+\alpha_{2}-\alpha_{1}.
	\end{equation*}
	Then 
	\begin{equation}
	\dot{K}_{p,m}^{\alpha_{2},r}(\Omega) \hookrightarrow \dot{K}%
	_{q}^{\alpha_{1},r}(\Omega),  \label{emb10}
	\end{equation}
	and 
	\begin{equation*}
	K_{p,m}^{\alpha_{2},r}(\Omega) \hookrightarrow K_{q}^{\alpha_{1},r}(\Omega)
	\end{equation*}
	hold.
\end{theorem}

\section{Pointwise Multiplication and the Rellich--Kondrachov Theorem}

A Banach space $A$ is called a \emph{Banach algebra} if it is equipped with
an associative multiplication satisfying 
\begin{equation*}
\|fg\|_{A} \lesssim \|f\|_{A}\,\|g\|_{A}
\end{equation*}
for all $f,g\in A$.

The following theorem shows that the Herz-type Sobolev  spaces form Banach algebras with respect to pointwise multiplication.

\begin{theorem}
	Let $\Omega \subset \mathbb{R}^{n}$ be a domain satisfying the cone
	condition, and let $m\in\mathbb{N}_{0}$. Assume that 
	$1<p<\infty$,$ 1\leq r<\infty,$
	$0\leq \alpha < n-\frac{n}{p}$ and $ m>\frac{n}{p}+\alpha.$
	Then  both 
	$
	K_{p,m}^{\alpha,r}(\Omega)
	$ and $
	\dot{K}_{p,m}^{\alpha,r}(\Omega)
	$ 
	are  Banach algebras with respect to pointwise multiplication.
\end{theorem}

\begin{proof}
	By similarity, it suffices to consider only the spaces $\dot{K}_{p,m}^{\alpha ,r}(\Omega )$.
	Let $f,g\in \dot{K}_{p,m}^{\alpha ,r}(\Omega )$. We aim to prove that 
	\begin{equation}
	\|fg\|_{\dot{K}_{p,m}^{\alpha ,r}(\Omega )} \lesssim \|f\|_{\dot{K}%
		_{p,m}^{\alpha ,r}(\Omega )} \|g\|_{\dot{K}_{p,m}^{\alpha ,r}(\Omega )}.\label{main-estimate1}
	\end{equation} 
	
	\textbf{Step 1.} Assume first that $f\in C^{\infty}(\Omega)$. By the Leibniz
	rule, 
	\begin{equation*}
	D^{\beta}(fg) = \sum_{\lambda \leq \beta} C_{\lambda}^{\beta} D^{\lambda}f\,
	D^{\beta-\lambda}g .
	\end{equation*}
	Therefore, it suffices to show that for every multi-index $\lambda \leq
	\beta $ with $|\beta|\leq m$, 
	\begin{equation}
	\|D^{\lambda}f\, D^{\beta-\lambda}g\|_{\dot{K}_{p}^{\alpha ,r}(\Omega )}
	\lesssim \|f\|_{\dot{K}_{p,m}^{\alpha ,r}(\Omega )} \|g\|_{\dot{K}%
		_{p,m}^{\alpha ,r}(\Omega )}.  \label{main-estimate}
	\end{equation}
	
	Set 
	\begin{equation*}
	k= 
	\begin{cases}
	m-\dfrac{n}{p}-\alpha -1, & \text{if } m-\dfrac{n}{p}-\alpha \in \mathbb{N},
	\\[1ex] 
	\left\lfloor m-\dfrac{n}{p}-\alpha \right\rfloor, & \text{if } m-\dfrac{n}{p}%
	-\alpha \notin \mathbb{N}.%
	\end{cases}%
	\end{equation*}
	
	\textbf{Substep 1.1.} We first prove \eqref{main-estimate} in the cases $%
	|\lambda|\leq k$ and $|\lambda|=m$. By Theorem~\ref{embeddingsq=infinity}, 
	\begin{equation}
	\dot{K}_{p,s}^{\alpha ,r}(\Omega ) \hookrightarrow L^{\infty}(\Omega )
	\label{emb1}
	\end{equation}
	for every integer $s>\dfrac{n}{p}+\alpha$.
	
	\textbf{Case 1.} $|\lambda|\leq k$. Clearly, 
	\begin{align*}
	\|D^{\lambda}f\, D^{\beta-\lambda}g\|_{\dot{K}_{p}^{\alpha,r}(\Omega)}
	&\lesssim \|D^{\lambda}f\|_{L^{\infty}(\Omega)} \|D^{\beta-\lambda}g\|_{\dot{%
			K}_{p}^{\alpha,r}(\Omega)} \\
	&\lesssim \|D^{\lambda}f\|_{L^{\infty}(\Omega)} \|g\|_{\dot{K}%
		_{p,m}^{\alpha,r}(\Omega)} .
	\end{align*}
	
	Since $|\lambda|\leq k$, we have 
	\begin{equation*}
	m-|\lambda|>\frac{n}{p}+\alpha .
	\end{equation*}
	Hence, by the embedding \eqref{emb1}, 
	\begin{equation*}
	\|D^{\lambda}f\|_{L^{\infty}(\Omega)} \lesssim \|D^{\lambda}f\|_{\dot{K}%
		_{p,m-|\lambda|}^{\alpha,r}(\Omega)} .
	\end{equation*}
	Moreover, 
	\begin{equation*}
	\|D^{\lambda}f\|_{\dot{K}_{p,m-|\lambda|}^{\alpha,r}(\Omega)} \leq \|f\|_{%
		\dot{K}_{p,m}^{\alpha,r}(\Omega)} .
	\end{equation*}
	Combining the above estimates yields \eqref{main-estimate} whenever $%
	|\lambda|\leq k$.
	
	\medskip
	
	\textbf{Case 2.} $|\lambda|=m$. Since $m>\frac{n}{p}+\alpha$, it follows
	from \eqref{emb1} that 
	\begin{align*}
	\|D^{\lambda}f\, g\|_{\dot{K}_{p}^{\alpha,r}(\Omega)} &\lesssim
	\|D^{\lambda}f\|_{\dot{K}_{p}^{\alpha,r}(\Omega)} \|g\|_{L^{\infty}(\Omega)}
	\\
	&\lesssim \|f\|_{\dot{K}_{p,m}^{\alpha,r}(\Omega)} \|g\|_{\dot{K}%
		_{p,m}^{\alpha,r}(\Omega)} .
	\end{align*}
	Therefore, \eqref{main-estimate} also holds whenever $|\lambda|=m$.
	
	\medskip
	
	\textbf{Substep 1.2.} We next prove \eqref{main-estimate} in the case 
	\begin{equation*}
	k<|\lambda|<m .
	\end{equation*}
	We distinguish between the two cases $|\beta-\lambda|\leq k$ and $%
	|\beta-\lambda|>k$.
	
	\medskip
	
	\textbf{Case 1.} $|\beta-\lambda|\leq k$. In this situation, 
	\begin{equation*}
	m-|\beta-\lambda|>\frac{n}{p}+\alpha .
	\end{equation*}
	Therefore, 
	\begin{align*}
	\|D^{\lambda}f\, D^{\beta-\lambda}g\|_{\dot{K}_{p}^{\alpha,r}(\Omega)}
	&\lesssim \|D^{\lambda}f\|_{\dot{K}_{p}^{\alpha,r}(\Omega)}
	\|D^{\beta-\lambda}g\|_{L^{\infty}(\Omega)} \\
	&\lesssim \|f\|_{\dot{K}_{p,m}^{\alpha,r}(\Omega)} \|D^{\beta-\lambda}g\|_{%
		\dot{K}_{p,m-|\beta-\lambda|}^{\alpha,r}(\Omega)} \\
	&\lesssim \|f\|_{\dot{K}_{p,m}^{\alpha,r}(\Omega)} \|g\|_{\dot{K}%
		_{p,m}^{\alpha,r}(\Omega)},
	\end{align*}
	where we again used the embedding \eqref{emb1}. Hence, \eqref{main-estimate}
	holds whenever $k<|\lambda|<m$ and $|\beta-\lambda|\leq k$.
	
	\medskip
	
	\textbf{Case 2.} $|\beta-\lambda|>k$. We divide the proof into two further
	subcases.
	
	\medskip
	
	\textbf{Case 2.1.} $m-\frac{n}{p}<|\lambda|<m .$ First observe that 
	\begin{equation*}
	m-|\beta-\lambda| = m-|\beta|+|\lambda| \geq |\lambda| > m-\frac{n}{p} >
	\alpha .
	\end{equation*}
	Consequently, 
	\begin{equation*}
	0<m-|\lambda|<\frac{n}{p}, \qquad \alpha<m-|\beta-\lambda| \leq \frac{n}{p}%
	+\alpha .
	\end{equation*}
	
	Choose $q>1$ such that 
	\begin{equation*}
	p\leq qp<\frac{n}{\frac{n}{p}-m+|\lambda|} \qquad \text{and} \qquad p\leq
	q^{\prime }p<\frac{n}{\frac{n}{p}+\alpha-m+|\beta-\lambda|},
	\end{equation*}
	where $\frac1q+\frac1{q^{\prime }}=1$. Such a choice is possible because 
	\begin{align*}
	\frac{p}{n} \Big( \frac{n}{p}-m+|\lambda| + \frac{n}{p}+\alpha-m+|\beta-%
	\lambda| \Big) &= 2-\frac{p}{n}(2m-\alpha-|\beta|) \\
	&\leq 2-\frac{p}{n}(m-\alpha) \\
	&< 1,
	\end{align*}
	since $m>\frac{n}{p}+\alpha$.
	
	By Theorems~\ref{embeddingsfirst copy(1)} and \ref{embedingsp<q}, we obtain 
	\begin{equation}
	\dot{K}_{p,m-|\lambda|}^{\alpha,r}(\Omega) \hookrightarrow \dot{K}%
	_{qp}^{\alpha,r}(\Omega), \qquad \dot{K}_{p,m-|\beta-\lambda|}^{\alpha,r}(%
	\Omega) \hookrightarrow \dot{K}_{q^{\prime }p}^{0,r}(\Omega).  \label{emb2}
	\end{equation}
	
	Using H\"{o}lder's inequality, for every $k\in\mathbb{Z}$ we have 
	\begin{align*}
	2^{k\alpha} \|(D^{\lambda}f\, D^{\beta-\lambda}g)\,\chi_{R_k\cap\Omega}\|_{p}
	&\lesssim 2^{k\alpha} \|(D^{\lambda}f)\,\chi_{R_k\cap\Omega}\|_{qp}
	\|(D^{\beta-\lambda}g)\,\chi_{R_k\cap\Omega}\|_{q^{\prime }p} \\
	&\lesssim \|(D^{\lambda}f)\|_{\dot{K}_{qp}^{\alpha,\infty}(\Omega)}
	\|(D^{\beta-\lambda}g)\,\chi_{R_k\cap\Omega}\|_{q^{\prime }p}.
	\end{align*}
	Therefore, 
	\begin{align*}
	\|D^{\lambda}f\, D^{\beta-\lambda}g\|_{\dot{K}_{p}^{\alpha,r}(\Omega)}
	&\lesssim \|D^{\lambda}f\|_{\dot{K}_{qp}^{\alpha,r}(\Omega)}
	\|D^{\beta-\lambda}g\|_{\dot{K}_{q^{\prime }p}^{0,r}(\Omega)} \\
	&\lesssim \|D^{\lambda}f\|_{\dot{K}_{p,m-|\lambda|}^{\alpha,r}(\Omega)}
	\|D^{\beta-\lambda}g\|_{\dot{K}_{p,m-|\beta-\lambda|}^{\alpha,r}(\Omega)} \\
	&\lesssim \|f\|_{\dot{K}_{p,m}^{\alpha,r}(\Omega)} \|g\|_{\dot{K}%
		_{p,m}^{\alpha,r}(\Omega)},
	\end{align*}
	where in the second inequality we used the embeddings \eqref{emb2}.
	
	\medskip
	
	\textbf{Case 2.2.} 
	$k<|\lambda|\leq m-\frac{n}{p}.$ In this situation, the argument becomes more delicate, and we distinguish
	the following four cases:
	
	\begin{enumerate}
		\item $m-\frac{n}{p}<|\beta-\lambda|\leq m$,
		
		\item $k<|\beta-\lambda|<m-\frac{n}{p}$,
		
		\item $|\beta-\lambda|=m-\frac{n}{p}$ and $\alpha<\frac{n}{p}$,
		
		\item $|\beta-\lambda|=m-\frac{n}{p}$ and $\alpha\geq\frac{n}{p}$.
	\end{enumerate}
	
	\textbf{Case 2.2.1.} $m-\frac{n}{p}<|\beta-\lambda|\leq m .$ First observe that 
	\begin{equation*}
	|\beta-\lambda| = |\beta|-|\lambda| \leq m-|\lambda|,
	\end{equation*}
	and 
	\begin{equation*}
	m-|\lambda| = m+|\beta-\lambda|-|\beta| \geq |\beta-\lambda| > m-\frac{n}{p}
	> \alpha .
	\end{equation*}
	Moreover, 
	\begin{equation*}
	\alpha < m-|\lambda| \leq \frac{n}{p}+\alpha, \qquad 0<m-|\beta-\lambda|<%
	\frac{n}{p}.
	\end{equation*}
	
	Choose $q>1$ such that 
	\begin{equation*}
	p\leq qp < \frac{n}{\frac{n}{p}-m+\alpha+|\lambda|} \qquad \text{and} \qquad
	p\leq q^{\prime }p < \frac{n}{\frac{n}{p}-m+|\beta-\lambda|},
	\end{equation*}
	where $\frac1q+\frac1{q^{\prime }}=1$.
	Such a choice is possible since $m>\frac{n}{p}+\alpha$ and 
	\begin{align*}
	\frac{p}{n} \Big( \frac{n}{p}-m+\alpha+|\lambda| + \frac{n}{p}%
	-m+|\beta-\lambda| \Big) &= 2-\frac{p}{n}(2m-\alpha-|\beta|) \\
	&< 1.
	\end{align*}
	
	Again, by Theorems~\ref{embeddingsfirst copy(1)} and \ref{embedingsp<q}, we
	obtain 
	\begin{equation}
	\dot{K}_{p,m-|\lambda|}^{\alpha,r}(\Omega) \hookrightarrow \dot{K}%
	_{qp}^{0,r}(\Omega), \qquad \dot{K}_{p,m-|\beta-\lambda|}^{\alpha,r}(%
	\Omega) \hookrightarrow \dot{K}_{q^{\prime }p}^{\alpha,r}(\Omega).
	\label{emb3}
	\end{equation}
	
	Applying H\"{o}lder's inequality, for every $k\in\mathbb{Z}$ we have 
	\begin{align*}
	2^{k\alpha} \|(D^{\lambda}f\, D^{\beta-\lambda}g)\,\chi_{R_k\cap\Omega}\|_{p}
	&\lesssim 2^{k\alpha} \|(D^{\lambda}f)\,\chi_{R_k\cap\Omega}\|_{qp}
	\|(D^{\beta-\lambda}g)\,\chi_{R_k\cap\Omega}\|_{q^{\prime }p} \\
	&\lesssim \|D^{\lambda}f\|_{\dot{K}_{qp}^{0,\infty}(\Omega)} \,2^{k\alpha}
	\|(D^{\beta-\lambda}g)\,\chi_{R_k\cap\Omega}\|_{q^{\prime }p}.
	\end{align*}
	
	Consequently, 
	\begin{align*}
	\|D^{\lambda}f\, D^{\beta-\lambda}g\|_{\dot{K}_{p}^{\alpha,r}(\Omega)}
	&\lesssim \|D^{\lambda}f\|_{\dot{K}_{qp}^{0,\infty}(\Omega)}
	\|D^{\beta-\lambda}g\|_{\dot{K}_{q^{\prime }p}^{\alpha,r}(\Omega)} \\
	&\lesssim \|D^{\lambda}f\|_{\dot{K}_{p,m-|\lambda|}^{\alpha,r}(\Omega)}
	\|D^{\beta-\lambda}g\|_{\dot{K}_{p,m-|\beta-\lambda|}^{\alpha,r}(\Omega)} \\
	&\lesssim \|f\|_{\dot{K}_{p,m}^{\alpha,r}(\Omega)} \|g\|_{\dot{K}%
		_{p,m}^{\alpha,r}(\Omega)},
	\end{align*}
	where we used the embeddings \eqref{emb3}. 
	
	\textbf{Case 2.2.2.} $k<|\beta-\lambda|<m-\frac{n}{p}.$
	Assume first that $\alpha>0$. Choose $\theta>0$ such that 
	\begin{equation*}
	\frac{\frac{n}{p}+\alpha-m+|\beta-\lambda|}{\alpha} < \theta < \min\Bigl(1,%
	\frac{m-|\lambda|}{\alpha}\Bigr).
	\end{equation*}
	Such a choice is possible since 
	\begin{equation*}
	|\beta-\lambda|+|\lambda| = |\beta| < m+m-\frac{n}{p}-\alpha .
	\end{equation*}
	Consequently, 
	\begin{equation*}
	m-|\beta-\lambda| > \frac{n}{p}+\alpha-\theta\alpha, \qquad m-|\lambda| >
	\alpha-(1-\theta)\alpha .
	\end{equation*}
	
	Hence, by Theorems~\ref{embedingsq=p} and \ref{embeddingsq=infinity}, 
	\begin{equation}
	\dot{K}_{p,m-|\lambda|}^{\alpha,r}(\Omega) \hookrightarrow \dot{K}%
	_{p}^{(1-\theta)\alpha,r}(\Omega), \qquad \dot{K}_{p,m-|\beta-%
		\lambda|}^{\alpha,r}(\Omega) \hookrightarrow \dot{K}_{\infty}^{\theta%
		\alpha,r}(\Omega).  \label{emb4}
	\end{equation}
	
	Therefore, for every $k\in\mathbb{Z}$, 
	\begin{align*}
	2^{k\alpha} \|(D^{\lambda}f\,D^{\beta-\lambda}g)\,\chi_{R_k\cap\Omega}\|_{p}
	&\lesssim 2^{k\alpha} \|(D^{\lambda}f)\,\chi_{R_k\cap\Omega}\|_{p}
	\|(D^{\beta-\lambda}g)\,\chi_{R_k\cap\Omega}\|_{\infty} \\
	&\lesssim \|D^{\lambda}f\|_{\dot{K}_{p}^{(1-\theta)\alpha,\infty}(\Omega)}
	\,2^{k\theta\alpha} \|(D^{\beta-\lambda}g)\,\chi_{R_k\cap\Omega}\|_{\infty}.
	\end{align*}
	Consequently, 
	\begin{align*}
	\|D^{\lambda}f\,D^{\beta-\lambda}g\|_{\dot{K}_{p}^{\alpha,r}(\Omega)}
	&\lesssim \|D^{\lambda}f\|_{\dot{K}_{p}^{(1-\theta)\alpha,\infty}(\Omega)}
	\|D^{\beta-\lambda}g\|_{\dot{K}_{\infty}^{\theta\alpha,r}(\Omega)} \\
	&\lesssim \|D^{\lambda}f\|_{\dot{K}_{p,m-|\lambda|}^{\alpha,r}(\Omega)}
	\|D^{\beta-\lambda}g\|_{\dot{K}_{p,m-|\beta-\lambda|}^{\alpha,r}(\Omega)} \\
	&\lesssim \|f\|_{\dot{K}_{p,m}^{\alpha,r}(\Omega)} \|g\|_{\dot{K}%
		_{p,m}^{\alpha,r}(\Omega)},
	\end{align*}
	where we used the embeddings \eqref{emb4}.
	
	If $\alpha=0$, we simply choose $\theta=0$.
	
	\medskip
	
	\textbf{Case 2.2.3.} 
	$|\beta-\lambda| = m-\frac{n}{p} $ and $ \alpha<\frac{n}{p}.$ Assume first that 
	\begin{equation*}
	|\lambda|<m-\frac{n}{p}.
	\end{equation*}
	Then, for every $k\in\mathbb{Z}$, 
	\begin{align*}
	2^{k\alpha} \|(D^{\lambda}f\,D^{\beta-\lambda}g)\,\chi_{R_k\cap\Omega}\|_{p}
	&\lesssim 2^{k\alpha} \|(D^{\lambda}f)\,\chi_{R_k\cap\Omega}\|_{\infty}
	\|(D^{\beta-\lambda}g)\,\chi_{R_k\cap\Omega}\|_{p} \\
	&\lesssim 2^{k\alpha} \|(D^{\lambda}f)\,\chi_{R_k\cap\Omega}\|_{\infty}
	\|D^{\beta-\lambda}g\|_{\dot{K}_{p}^{0,\infty}(\Omega)}.
	\end{align*}
	Hence, 
	\begin{align*}
	\|D^{\lambda}f\,D^{\beta-\lambda}g\|_{\dot{K}_{p}^{\alpha,r}(\Omega)}
	&\lesssim \|D^{\lambda}f\|_{\dot{K}_{\infty}^{\alpha,\infty}(\Omega)}
	\|D^{\beta-\lambda}g\|_{\dot{K}_{p}^{0,\infty}(\Omega)} \\
	&\lesssim \|f\|_{\dot{K}_{p,m}^{\alpha,r}(\Omega)} \|g\|_{\dot{K}%
		_{p,m}^{\alpha,r}(\Omega)},
	\end{align*}
	where we used the embeddings 
	\begin{equation*}
	\dot{K}_{p,m-|\lambda|}^{\alpha,r}(\Omega) \hookrightarrow \dot{K}%
	_{\infty}^{\alpha,\infty}(\Omega), \qquad \dot{K}_{p,\frac{n}{p}%
	}^{\alpha,r}(\Omega) \hookrightarrow \dot{K}_{p}^{0,\infty}(\Omega),
	\end{equation*}
	which follow from Theorems~\ref{embedingsq=p} and \ref{embeddingsq=infinity}.
	
	Now assume that 
	\begin{equation*}
	|\lambda| = m-\frac{n}{p}.
	\end{equation*}
	Choose $q>1$ such that 
	\begin{equation*}
	p\leq qp<\frac{n}{\alpha-\alpha_1}, \qquad p\leq q^{\prime }p<\frac{n}{%
		\alpha-\alpha_2},
	\end{equation*}
	where 
	\begin{equation*}
	\frac1q+\frac1{q^{\prime }}=1, \qquad \alpha_1+\alpha_2=\alpha,\quad\alpha_1, \alpha_2\geq 0 .
	\end{equation*}
	Such a choice is possible because 
	\begin{equation*}
	\frac{p}{n} \bigl( \alpha-\alpha_1+\alpha-\alpha_2 \bigr) = \frac{p}{n}%
	\alpha < 1,
	\end{equation*}
	since $\alpha<\frac{n}{p}$.
	
	By Theorem~\ref{embeddingsfirst copy(1)}, 
	\begin{equation}
	\dot{K}_{p,\frac{n}{p}}^{\alpha,r}(\Omega) \hookrightarrow \dot{K}%
	_{qp}^{\alpha_1,r}(\Omega), \qquad \dot{K}_{p,\frac{n}{p}}^{\alpha,r}(%
	\Omega) \hookrightarrow \dot{K}_{q^{\prime }p}^{\alpha_2,\infty}(\Omega).
	\label{emb6}
	\end{equation}
	
	We now argue exactly as in Case~2.2.1, using the embeddings \eqref{emb6}
	instead of \eqref{emb3}.
	
	\medskip
	
	\textbf{Case 2.2.4.} 
$
	|\beta-\lambda| = m-\frac{n}{p}$ and $ \alpha\geq\frac{n}{p}.$ Then 
	\begin{equation*}
	|\lambda| = |\beta|-m+\frac{n}{p} < m-\alpha .
	\end{equation*}
	Moreover, for every $k\in\mathbb{Z}$, 
	\begin{align*}
	2^{k\alpha} \|(D^{\lambda}f\,D^{\beta-\lambda}g)\,\chi_{R_k\cap\Omega}\|_{p}
	&\lesssim 2^{k\alpha} \|(D^{\lambda}f)\,\chi_{R_k\cap\Omega}\|_{\infty}
	\|(D^{\beta-\lambda}g)\,\chi_{R_k\cap\Omega}\|_{p} \\
	&\lesssim \|D^{\lambda}f\|_{\dot{K}_{\infty}^{\frac{n}{p},\infty}(\Omega)}
	\,2^{k(\alpha-\frac{n}{p})} \|(D^{\beta-\lambda}g)\,\chi_{R_k\cap\Omega}\|_{p}.
	\end{align*}
	Hence, 
	\begin{equation*}
	\|D^{\lambda}f\,D^{\beta-\lambda}g\|_{\dot{K}_{p}^{\alpha,r}(\Omega)}
	\lesssim \|D^{\lambda}f\|_{\dot{K}_{\infty}^{\frac{n}{p},\infty}(\Omega)}
	\|D^{\beta-\lambda}g\|_{\dot{K}_{p}^{\alpha-\frac{n}{p},\infty}(\Omega)}.
	\end{equation*}
	
	By Theorems~\ref{embedingsq=p} and \ref{embeddingsq=infinity}, 
	\begin{equation*}
	\dot{K}_{p,m-|\lambda|}^{\alpha,\infty}(\Omega) \hookrightarrow \dot{K}%
	_{\infty}^{\frac{n}{p},\infty}(\Omega), \qquad \dot{K}_{p,\frac{n}{p}%
	}^{\alpha,\infty}(\Omega) \hookrightarrow \dot{K}_{p}^{\alpha-\frac{n}{p}%
		,\infty}(\Omega).
	\end{equation*}
	This yields the desired estimate.
	
	Therefore, \eqref{main-estimate} holds for all 
	\begin{equation*}
	k<|\lambda|<m.
	\end{equation*}
	
	\textbf{Step 2}. We now prove \eqref{main-estimate}. Arguing as in \cite{AdamsFournier03}, Theorem 4.39, and using the result established in Step 1, we obtain the desired estimate.\
	This completes the proof.
\end{proof}

For a function $f$ defined almost everywhere on $\Omega$, we denote by $%
\tilde f$ its zero extension outside $\Omega$, namely 
\begin{equation*}
\tilde f(x)= 
\begin{cases}
f(x), & \text{if } x\in\Omega, \\ 
0, & \text{if } x\in\mathbb{R}^{n}\setminus\Omega.%
\end{cases}%
\end{equation*}

We first establish the following elementary fact.

\begin{lemma}
	\label{translation1} Let $\Omega\subset\mathbb{R}^{n}$ be an open set. If $%
	1\leq p,q\leq\infty$, $\alpha\in\mathbb{R}$, and $h\in\mathbb{R}^{n}$ is
	sufficiently small, then the translation operator $\tau_h$ is bounded on $%
	K_{p}^{\alpha,q}(\Omega)$. More precisely, 
	\begin{equation}
	\|\tau_h\varphi\|_{K_{p}^{\alpha,q}(\Omega)} \lesssim
	\|\varphi\|_{K_{p}^{\alpha,q}(\Omega)}.  \label{translation3}
	\end{equation}
\end{lemma}

\begin{proof}
	Since $h$ is sufficiently small, we have $\Omega=\Omega+h.$ Therefore, it suffices to prove \eqref{translation3} in the special case $%
	\Omega=\mathbb{R}^{n}$. By similarity we can assume that $ 1\leq q<\infty$. We write 
	\begin{equation*}
	\|\tau_h\varphi\|_{K_{p}^{\alpha,q}(\mathbb{R}^{n})}^{q} = I_1+I_2,
	\end{equation*}
	where 
	\begin{align*}
	I_1 &= \|\tau_h\varphi\,\chi_{B_0}\|_{L^{p}(\mathbb{R}^{n})}^{q}, \qquad I_2 = \sum_{\substack{ k\in\mathbb{N}  \\ 2^{k}\geq 2|h|}} 2^{k\alpha q}
	\|\tau_h\varphi\,\chi_{R_k}\|_{L^{p}(\mathbb{R}^{n})}^{q}.
	\end{align*}
	
	Observe that 
	\begin{equation*}
	I_1^{1/q} \leq \|\varphi\,\chi_{B_2}\|_{L^{p}(\mathbb{R}^{n})} \lesssim
	\|\varphi\|_{K_{p}^{\alpha,q}(\mathbb{R}^{n})}.
	\end{equation*}
	Moreover, 
	\begin{align*}
	I_2 &\leq \sum_{\substack{ k\in\mathbb{N}  \\ 2^{k}\geq 2|h|}} 2^{k\alpha q}
	\|\varphi\,\chi_{\widetilde R_k}\|_{L^{p}(\mathbb{R}^{n})}^{q} 
	\lesssim \|\varphi\|_{K_{p}^{\alpha,q}(\mathbb{R}^{n})}^{q},
	\end{align*}
	where 
	\begin{equation*}
	\widetilde R_k = \{x\in\mathbb{R}^{n}:2^{k-2}\leq |x|\leq 2^{k+1}\}.
	\end{equation*}
	Combining the above estimates completes the proof.
\end{proof}

\begin{remark}
	\label{translation1 copy(1)} For arbitrary $h\in\mathbb{R}^{n}$, one has 
	\begin{equation*}
	\|\tau_h\varphi\|_{K_{p}^{\alpha,q}(\mathbb{R}^{n})} \lesssim
	(1+|h|)^{|\alpha|} \|\varphi\|_{K_{p}^{\alpha,q}(\mathbb{R}^{n})}
	\end{equation*}
	for every $1\leq p,q\leq\infty$ and $\alpha\in\mathbb{R}$, $\alpha\neq0$;
	see Lemma~4.1 in \cite{Fi-We08}.
\end{remark}

The following convolution inequality plays an essential role in our arguments.
\begin{lemma}
	\label{Key-est1}
	\textit{Let } $x\in \mathbb{R}^{n}$, $1\leq p\leq \infty$, 
	$0\leq \alpha < n-\frac{n}{p}$, \textit{and} $R>0$. 
	\textit{Then there exists a constant} $c>0$, 
	\textit{independent of} $R$, \textit{such that for every}
	$f\in K_{p}^{\alpha,\infty}$,
	\textit{the estimate}
	\begin{equation}
	\eta_{R,N}\ast |f|(x)
	\leq
	c\,\max\left(R^{\frac{n}{p}},
	R^{\frac{n}{p}+\alpha}\right)
	\|f\|_{K_{p}^{\alpha,\infty}}
	\label{convolution}
	\end{equation}
	\textit{holds whenever} $N\in\mathbb{N}$ \textit{is sufficiently large, where} $\eta_{R,N}=R^{n}(1+R|\cdot|)^{-N}$.
\end{lemma}
\begin{proof}
	We proceed in two steps.
	
	\textbf{Step 1.} In this step, we prove \eqref{convolution} for any
	$x\in B(0,\frac1R)$. Write $\eta_{R,N}\ast |f|=I_1+I_2,$
	where
	\[
	I_1=\eta_{R,N}\ast \chi_{B(0,\frac4R)}|f|,
	\qquad
	I_2=\eta_{R,N}\ast
	\chi_{\mathbb{R}^n\setminus B(0,\frac4R)}|f|.
	\]
	We estimate each term separately. We begin with $I_1$.
	
	First, assume that $R\leq 4$. Then
	\[
	I_1=
	\eta_{R,N}\ast \chi_{B_0}|f|
	+\eta_{R,N}\ast
	\chi_{\{1\leq |\cdot|\leq \frac4R\}}|f|
	=:
	I_{1,1}+I_{1,2}.
	\]
	
	By  H\"{o}lder's inequality,
	\[
	I_{1,1}
	\leq
	\|\eta_{R,N}\|_{p'}
	\|f\chi_{B_0}\|_p
	\leq
	R^{\frac np}
	\|f\chi_{B_0}\|_p
	\leq
	R^{\frac np}
	\|f\|_{K_p^{\alpha,\infty}}.
	\]
	
	Let $l\in\mathbb N$ satisfy $2^{l-1}\leq \frac4R<2^l.$
	Again, by  H\"{o}lder's inequality,
	\begin{align*}
	I_{1,2}
	&\leq
	\sum_{v=1}^{l}
	\eta_{R,N}\ast \chi_{R_v}|f|\leq
	R^n
	\sum_{v=1}^{l}
	\|f\chi_{R_v}\|_1\lesssim
	R^n
	\sum_{v=1}^{l}
	2^{v\frac{n}{p'}}
	\|f\chi_{R_v}\|_p,
	\end{align*}
	which is bounded by
	\[
	cR^n
	\sum_{v=1}^{l}
	2^{v(\frac{n}{p'}-\alpha)}
	\|f\|_{K_p^{\alpha,\infty}}
	\lesssim
	R^{\frac np+\alpha}
	\|f\|_{K_p^{\alpha,\infty}},
	\]
	where we used the assumption
	$\alpha<n-\frac np$, and the implicit constant is independent of $R$.
	
	Now suppose that $R>4$. Then
	\[
	I_1
	\leq
	\eta_{R,N}\ast \chi_{B_0}|f|
	\lesssim
	R^{\frac np}
	\|f\|_{K_p^{\alpha,\infty}}.
	\]
	
	Next, we estimate $I_2$. Assume first that $R\leq4$.
	By  H\"{o}lder's inequality,
	\begin{align*}
	I_2(x)
	&\leq
	\sum_{v=l}^{\infty}
	\eta_{R,N}\ast \chi_{R_v}|f|(x)
	\leq
	R^{n-N}
	\sum_{v=l}^{\infty}
	2^{-Nv}
	\|f\chi_{R_v}\|_1
	\leq
	R^{n-N}
	\sum_{v=l}^{\infty}
	2^{v(\frac{n}{p'}-N)}
	\|f\chi_{R_v}\|_p.
	\end{align*}
	
	Hence,
	\begin{align*}
	I_2
	&\lesssim
	R^{n-N}
	\sum_{v=l}^{\infty}
	2^{v(-N+\frac{n}{p'}-\alpha)}
	\|f\|_{K_p^{\alpha,\infty}}
	\lesssim
	R^{\frac np+\alpha}
	\|f\|_{K_p^{\alpha,\infty}},
	\end{align*}
	provided that $N$ is sufficiently large.
	
	Now assume that $R>4$. Write $I_2=I_{2,3}+I_{2,4},$
	where
	\[
	I_{2,3}
	=
	\eta_{R,N}\ast
	\chi_{\{\frac4R\leq |\cdot|\leq1\}}|f|,
	\qquad
	I_{2,4}
	=
	\eta_{R,N}\ast
	\chi_{\{|\cdot|>1\}}|f|.
	\]
	
	Obviously,
	\[
	I_{2,3}
	\leq
	\eta_{R,N}\ast \chi_{B_0}|f|
	\lesssim
	R^{\frac np}
	\|f\|_{K_p^{\alpha,\infty}},
	\]
	and
	\begin{align*}
	I_{2,4}(x)
	&=
	\sum_{v=1}^{\infty}
	\eta_{R,N}\ast \chi_{R_v}|f|(x)
	\\
	&\lesssim
	R^{n-N}
	\sum_{v=1}^{\infty}
	2^{v(-N+\frac{n}{p'}-\alpha)}
	\|f\|_{K_p^{\alpha,\infty}}
	\\
	&\lesssim
	R^{\frac np+\alpha}
	\|f\|_{K_p^{\alpha,\infty}},
	\end{align*}
	since $N$ is sufficiently large, $\alpha\ge0$, and in this case
	$R>4$.
	
	\textbf{Step 2.} In this step, we prove \eqref{convolution} for any
	$x\notin B(0,\frac{1}{R})$.
	
	\textbf{Substep 2.1.} Assume that
	$\frac{1}{R}\leq |x|\leq 1$. We write $\eta_{R,N}\ast |f|=F_{1}+F_{2},$
	where
	\[
	F_{1}=\eta_{R,N}\ast \chi_{\bar{B_{2}}}|f|,
	\qquad
	F_{2}=\eta_{R,N}\ast \chi_{\{|\cdot |>2\}}|f|.
	\]
	
	Observe that
	\[
	F_{1}\lesssim
	\max\bigl(R^{\frac{n}{p}},R^{\frac{n}{p}+\alpha}\bigr)
	\|f\|_{K_{p}^{\alpha,\infty}}.
	\]
	By  H\"{o}lder's inequality, we obtain
	\begin{align*}
	F_{2}(x)
	&=
	\sum_{v=2}^{\infty}
	\eta_{R,N}\ast \chi_{R_v}|f|(x)
	\\
	&\lesssim
	R^{n-N}
	\sum_{v=2}^{\infty}
	2^{-Nv}
	\|f\chi_{R_v}\|_{1}
	\\
	&\lesssim
	R^{n-N}
	\sum_{v=2}^{\infty}
	2^{v(\frac{n}{p'}-N)}
	\|f\chi_{R_v}\|_{p}.
	\end{align*}
	Hence,
	\begin{align*}
	F_{2}
	&\lesssim
	R^{n-N}
	\sum_{v=2}^{\infty}
	2^{v(\frac{n}{p'}-\alpha-N)}
	\|f\|_{K_{p}^{\alpha,\infty}}
	\\
	&\lesssim
	R^{\frac{n}{p}+\alpha}
	\|f\|_{K_{p}^{\alpha,\infty}},
	\end{align*}
	provided that $N$ is chosen sufficiently large. Since
	$\alpha\geq 0$ and $R\geq 1$ in this case, the desired estimate follows.
	
	\textbf{Substep 2.2.} Assume that $|x|>1$. Let $k\in \mathbb{N}$ be such
	that $2^{k-1}\leq |x|<2^{k}$. We write
	$\eta_{R,N}\ast |f|=M_{1}+M_{2}+M_{3}$, where
	\begin{equation*}
	M_{1}=\eta_{R,N}\ast \chi_{B_{k-2}}|f|,
	\qquad
	M_{2}=\eta_{R,N}\ast \chi_{\breve{R}_{k}}|f|,
	\end{equation*}
	with
	$
	\breve{R}_{k}
	=
	\{y\in\mathbb{R}^{n}:2^{k-2}\le |y|<2^{k+2}\},
	$
	and
	\begin{equation*}
	M_{3}
	=
	\eta_{R,N}\ast \chi_{\{|\cdot|\ge 2^{k+2}\}}|f|.
	\end{equation*}
	
	By  H\"{o}lder's inequality,
	\begin{align*}
	M_{2}
	&\le
	\|\eta_{R,N}\|_{p'}\,
	\|\chi_{\breve{R}_{k}}f\|_{p}
	\\
	&\lesssim
	R^{\frac{n}{p}}
	\|\chi_{\breve{R}_{k}}f\|_{p}
	\\
	&\lesssim
	R^{\frac{n}{p}}
	2^{-\alpha k}
	\|f\|_{K_p^{\alpha,\infty}}
	\\
	&\lesssim
	R^{\frac{n}{p}}
	\|f\|_{K_p^{\alpha,\infty}},
	\end{align*}
	since $k\in\mathbb N$ and $\alpha\ge0$.
	
	Now
	\[
	M_{1}
	=
	\eta_{R,N}\ast \chi_{B_{0}}|f|
	+
	\eta_{R,N}\ast
	\chi_{\{1\le |\cdot|<2^{k-2}\}}|f|.
	\]
	We have
	\begin{align*}
	\eta_{R,N}\ast
	\chi_{\{1\le |\cdot|<2^{k-2}\}}|f|
	&\le
	R^n(R2^k)^{-N}
	\big\|
	\chi_{\{1\le |\cdot|<2^{k-2}\}}f
	\big\|_1
	\\
	&=
	R^n(R2^k)^{-N}
	\sum_{l=1}^{k-2}
	\|\chi_l f\|_1
	\\
	&\lesssim
	R^n(R2^k)^{-N}
	\sum_{l=1}^{k-2}
	2^{(\frac{n}{p'}-\alpha)l}
	2^{\alpha l}
	\|\chi_l f\|_p
	\\
	&\lesssim
	R^{n-N}
	2^{-kN}
	\sum_{l=1}^{k-2}
	2^{(\frac{n}{p'}-\alpha)l}
	\|f\|_{K_p^{\alpha,\infty}}.
	\end{align*}
	
	Since $\alpha<n-\frac{n}{p}=\frac{n}{p'}$, we have
	$\frac{n}{p'}-\alpha>0$, and hence
	\[
	\sum_{l=1}^{k-2}
	2^{(\frac{n}{p'}-\alpha)l}
	\lesssim
	2^{(\frac{n}{p'}-\alpha)k}.
	\]
	Therefore,
	\[
	\eta_{R,N}\ast
	\chi_{\{1\le |\cdot|<2^{k-2}\}}|f|
	\lesssim
	R^{n-N}
	2^{k(\frac{n}{p'}-\alpha-N)}
	\|f\|_{K_p^{\alpha,\infty}}.
	\]
	Hence
	\[
	\eta_{R,N}\ast
	\chi_{\{1\le |\cdot|<2^{k-2}\}}|f|
	\lesssim
	R^{\frac{n}{p}+\alpha}
	(R2^k)^{\frac{n}{p'}-\alpha-N}
	\|f\|_{K_p^{\alpha,\infty}}
	\lesssim
	R^{\frac{n}{p}+\alpha}
	\|f\|_{K_p^{\alpha,\infty}},
	\]
	since $R2^k\ge1$ and $N$ is chosen sufficiently large.
	
	Finally,
	\begin{equation*}
	M_{3}(x)
	=
	\sum_{l=k+3}^{\infty}
	\eta_{R,N}\ast\chi_{R_l}|f|(x)
	\le
	R^{n-N}
	\sum_{l=k+3}^{\infty}
	2^{-Nl}
	\|\chi_l f\|_1,
	\end{equation*}
	which is bounded by
	\begin{align*}
	&R^{n-N}
	\sum_{l=k+3}^{\infty}
	2^{(\frac{n}{p'}-\alpha-N)l}
	2^{\alpha l}
	\|\chi_l f\|_p
	\\
	&\lesssim
	R^{n-N}
	\sum_{l=k+3}^{\infty}
	2^{(\frac{n}{p'}-\alpha-N)l}
	\|f\|_{K_p^{\alpha,\infty}}
	\\
	&\lesssim
	R^{\frac{n}{p}+\alpha}
	\|f\|_{K_p^{\alpha,\infty}}.
	\end{align*}
	The proof is complete.
	
\end{proof}
Similarly to Lemma \ref{Key-est1}, we obtain the following result.

\begin{lemma}
	\label{Key-est2}
	\textit{Let } $x\in\mathbb{R}^{n}$, $1\leq p\leq\infty$, \textit{and } $R>0$.
	\textit{Then there exists a constant } $c>0$, \textit{independent of } $R$,
	\textit{such that for all } $f\in K_{p}^{\,n-\frac{n}{p},1}$,
	\textit{we have}
	\begin{equation*}
	\eta_{R,N}*|f|(x)
	\leq
	c\max\left(R^{\frac{n}{p}},
	R^{n}\right)
	\|f\|_{K_{p}^{\,n-\frac{n}{p},1}}
	\end{equation*}
	\textit{for any sufficiently large } $N\in\mathbb{N}$.
\end{lemma}

\begin{remark}
	\label{Key-est3}
	(i) Let $x\in\mathbb{R}^{n}$, $1<p\leq\infty$, $\alpha\geq0$, and $R>0$.
	By Lemma \ref{Key-est1} and \eqref{herz-emb1}, there exists a constant $c>0$, depending on $R$, such that for every
	$f\in K_{p}^{\,\alpha,\infty}$,
	\begin{equation*}
	\eta_{R,N}*|f|(x)
	\leq
	c\,\|f\|_{K_{p}^{\,\alpha,\infty}}
	\end{equation*}
	for all sufficiently large $N\in\mathbb{N}$.
	
	(ii) Let $x\in\mathbb{R}^{n}$, $\alpha\geq0$, and $R>0$.
	By Lemma \ref{Key-est2} and \eqref{herz-emb1}, there exists a constant $c>0$, depending on $R$, such that for every
	$f\in K_{1}^{\,\alpha,r}$,
	\begin{equation*}
	\eta_{R,N}*|f|(x)
	\leq
	c\,\|f\|_{K_{1}^{\,\alpha,r}}
	\end{equation*}
	for all sufficiently large $N\in\mathbb{N}$, where
	\begin{equation*}
	r=
	\begin{cases}
	\infty, & \text{if }\alpha>0,\\
	1, & \text{if }\alpha=0.
	\end{cases}
	\end{equation*}
\end{remark}
Recall that a subset $A$ of a normed space $X$ is called \emph{precompact}
if its closure $\overline A$ is compact in the norm topology of $X$.

Equivalently, a subset $A\subset X$ is precompact if and only if for every $%
\varepsilon>0$ there exists a finite subset $\{x_1,\dots,x_N\}\subset X$
such that 
\begin{equation*}
A\subset \bigcup_{j=1}^{N} B(x_j,\varepsilon).
\end{equation*}

We now prove a compactness criterion in Herz spaces.

\begin{theorem}
	\label{precompact} Let $\Omega\subset\mathbb{R}^{n}$ be an open set $1<p<\infty$, $1\leq
	q<\infty$ and $ \alpha \geq 0$. Suppose that $A\subset
	K_{p}^{\alpha,q}(\Omega)$ is bounded. Then the following assertions are
	equivalent:
	
	\begin{enumerate}
		\item[(i)] $A$ is precompact in $K_{p}^{\alpha,q}(\Omega)$;
		
		\item[(ii)] For every $\varepsilon>0$ there exist $\delta>0$ and a 
		subset $G\Subset\Omega$ such that for every $f\in A$ and every $h\in\mathbb{R}%
		^{n}$ with $|h|<\delta$, 
		\begin{equation}
		\|\tilde f(\,\cdot+h)-\tilde f\|_{K_{p}^{\alpha,q}(\Omega)} < \varepsilon
		\label{prop3}
		\end{equation}
		and 
		\begin{equation}
		\|f\|_{K_{p}^{\alpha,q}(\Omega\setminus\overline G)} < \varepsilon .
		\label{pro2}
		\end{equation}
	\end{enumerate}
\end{theorem}

\begin{proof}
	We proceed in two steps.
	
	\textit{Step 1.} Assume that $A$ is precompact in $K_{p}^{\alpha,q}(\Omega) $%
	. Then, for every $\varepsilon>0$, there exists a finite set $%
	N_{\varepsilon}\subset K_{p}^{\alpha,q}(\Omega)$ such that 
	\begin{equation*}
	A \subset \bigcup_{f\in N_{\varepsilon}} B\Bigl(f,\frac{\varepsilon}{6}\Bigr)%
	,
	\end{equation*}
	where 
	\begin{equation*}
	B\Bigl(f,\frac{\varepsilon}{6}\Bigr) = \Bigl\{ g\in
	K_{p}^{\alpha,q}(\Omega): \|g-f\|_{K_{p}^{\alpha,q}(\Omega)} < \frac{\varepsilon}{6} \Bigr\}.
	\end{equation*}
	
	Since $C_{c}(\Omega)$ is dense in $K_{p}^{\alpha,q}(\Omega)$ (Theorem~\ref{dense}),
	there exists a finite set $S\subset C_{c}(\Omega)$ such that for every $f\in
	A$ there exists $\varphi\in S$ satisfying 
	\begin{equation*}
	\|f-\varphi\|_{K_{p}^{\alpha,q}(\Omega)} < \frac{\varepsilon}{3}%
	.
	\end{equation*}
	
	Let $G$ denote the union of the supports of the finitely many functions in $%
	S $. Then \eqref{pro2} follows immediately.
	
	Let $\overline{B(0,r)}$ be a closed ball containing $G$. Since each $%
	\varphi\in S$ has compact support in $G$, we obtain 
	\begin{align}
	\|\tau_h\varphi-\varphi\|_{K_{p}^{\alpha,q}(\mathbb{R}^{n})}^{q} &=
	\sum_{k=0}^{\infty} 2^{k\alpha q} \|(\tau_h\varphi-\varphi)\chi_{
		R_k}\|_{L^{p}(\mathbb{R}^{n})}^{q}  \notag \\
	&= \sum_{\substack{ k\in\mathbb{N}  \\ 2^{k}\leq 2r+2}} 2^{k\alpha q}
	\|(\tau_h\varphi-\varphi) \chi_{ R_k\cap B(0,r+1)}\|_{L^{p}(\mathbb{R%
		}^{n})}^{q}.  \label{serie}
	\end{align}
	
	Moreover, 
	\begin{equation*}
	|(\tau_h\varphi-\varphi)\chi_{\widehat R_k\cap B(0,r+1)}| \leq
	2\|\varphi\|_{\infty} \chi_{ R_k\cap B(0,r+1)},
	\end{equation*}
	and therefore 
	\begin{equation*}
	\|(\tau_h\varphi-\varphi) \chi_{ R_k\cap B(0,r+1)}\|_{L^{p}(\mathbb{R%
		}^{n})} \lesssim 2^{k\frac np}\|\varphi\|_{\infty}.
	\end{equation*}
	
	Since only finitely many indices $k$ appear in \eqref{serie}, the dominated
	convergence theorem implies that 
	\begin{equation*}
	\lim_{|h|\to0} \|(\tau_h\varphi-\varphi) \chi_{\widehat R_k\cap
		B(0,r+1)}\|_{L^{p}(\mathbb{R}^{n})} = 0.
	\end{equation*}
	Consequently, 
	\begin{equation*}
	\lim_{|h|\to0} \|\tau_h\varphi-\varphi\|_{K_{p}^{\alpha,q}(\mathbb{R}^{n})}
	= 0.
	\end{equation*}
	
	Hence, for sufficiently small $|h|$, 
	\begin{equation*}
	\|\tau_h\varphi-\varphi\|_{K_{p}^{\alpha,q}(\mathbb{R}^{n})} < \frac{%
		\varepsilon}{3}.
	\end{equation*}
	
	Furthermore, by Lemma~\ref{translation1}, 
	\begin{equation*}
	\|\tau_h\tilde f-\tau_h\varphi\|_{K_{p}^{\alpha,q}(\Omega)} \lesssim
	\|\tilde f-\varphi\|_{K_{p}^{\alpha,q}(\Omega)} < \frac{\varepsilon}{3}.
	\end{equation*}
	
	Therefore, 
	\begin{align*}
	\|\tau_h\tilde f-\tilde f\|_{K_{p}^{\alpha,q}(\Omega)} &\leq \|\tau_h\tilde
	f-\tau_h\varphi\|_{K_{p}^{\alpha,q}(\Omega)} 
	+ \|\tau_h\varphi-\varphi\|_{K_{p}^{\alpha,q}(\Omega)} +
	\|f-\varphi\|_{K_{p}^{\alpha,q}(\Omega)} \\
	&< \varepsilon .
	\end{align*}
	
	This proves \eqref{prop3}.
	
	\textit{Step 2.} As in \cite[Theorem~2.32]%
	{AdamsFournier03}, it suffices to prove the converse in the special case $\Omega=\mathbb{R}^{n}.$
	
	Let $\varepsilon>0$ be given. By assumption, there exists a bounded set $%
	G\Subset\mathbb{R}^{n}$ such that 
	\begin{equation}
	\|f\|_{K_{p}^{\alpha,q}(\mathbb{R}^{n}\setminus\overline G)} < \frac{%
		\varepsilon}{3}  \label{tail-est}
	\end{equation}
	for every $f\in A$.
	
	Let $J_{\varepsilon}$ be a standard mollifier; see Section 3. Arguing as in \cite[%
	Theorem~2.32]{AdamsFournier03}, we obtain 
	\begin{equation*}
	\lim_{\varrho\to0} \|J_{\varrho}*f-f\|_{K_{p}^{\alpha,q}(\mathbb{R}^{n})} = 0
	\end{equation*}
	uniformly for $f\in A$. Hence, for sufficiently small $\varrho>0$, 
	\begin{equation}
	\|J_{\varrho}*f-f\|_{K_{p}^{\alpha,q}(\mathbb{R}^{n})} < \frac{\varepsilon}{3%
	}  \label{moll-est}
	\end{equation}
	for all $f\in A$.
	
	We now show that 
	\begin{equation*}
	\mathcal{F}_{\varrho} = \{J_{\varrho}*f:\ f\in A\}
	\end{equation*}
	is precompact in $C(\overline G)$. By the Arzela-Ascoli theorem (see \cite[%
	Theorem~1.33]{AdamsFournier03}), it suffices to prove that $\mathcal{F}%
	_{\varrho}$ is uniformly bounded and equicontinuous on $\overline G$.
	
	First, 
	\begin{equation*}
	|J_{\varrho}*f(x)| \lesssim \|f\|_{K_{p}^{\alpha,q}(\mathbb{R}^{n})},
	\end{equation*}
	where the implicit constant is independent of $f\in A$ and $x\in\mathbb{R}%
	^{n}$; see, Lemma 	\ref{Key-est1} and  Remark \ref{Key-est3}. Since $A$ is bounded in $K_{p}^{\alpha,q}(\mathbb{R}^{n})$, the
	family $\mathcal{F}_{\varrho}$ is uniformly bounded.
	
	Moreover, 
	\begin{equation*}
	|J_{\varrho}*f(x+h)-J_{\varrho}*f(x)|=|J_{\varrho}*(\tau_hf-f)(x)|\lesssim \|\tau_h
	f-f\|_{K_{p}^{\alpha,q}(\mathbb{R}^{n})},
	\end{equation*}
	where the implicit constant is independent of $f$, $x$, and $h$; see, Lemma \ref{Key-est1} and Remark	\ref{Key-est3}. By assumption, 
	\begin{equation*}
	\lim_{|h|\to0} \|\tau_h f-f\|_{K_{p}^{\alpha,q}(\mathbb{R}^{n})} = 0
	\end{equation*}
	uniformly for $f\in A$. Therefore, 
	\begin{equation*}
	\lim_{|h|\to0} |J_{\varrho}*f(x+h)-J_{\varrho}*f(x)| = 0
	\end{equation*}
	uniformly with respect to $x\in\mathbb{R}^{n}$ and $f\in A$. Thus $\mathcal{F%
	}_{\varrho}$ is equicontinuous on $\overline G$.
	
	Hence $\mathcal{F}_{\varrho}$ is precompact in $C(\overline G)$.
	Consequently, there exist finitely many functions 
	\begin{equation*}
	\omega_1,\dots,\omega_m\in C(\overline G)
	\end{equation*}
	such that for every $f\in A$ there exists $j\in\{1,\dots,m\}$ satisfying 
	\begin{equation}
	|J_{\varrho}*f(x)-\omega_j(x)| < \frac{\varepsilon}{3F} \qquad \text{for all 
	}x\in\overline G,  \label{uniform-approx}
	\end{equation}
	where 
	\begin{equation*}
	F^{q} = \sum_{k=0}^{\infty} 2^{k\alpha q} |R_k\cap\overline B_v|^{q/p},
	\end{equation*}
	and $\overline G\subset\overline B_v$ for some $v\in\mathbb{N}$.
	
	We now estimate $\|f-\omega_j\|_{K_{p}^{\alpha,q}(\mathbb{R}^{n})}$. Using %
	\eqref{tail-est}, we obtain 
	\begin{align*}
	\|f-\omega_j\|_{K_{p}^{\alpha,q}(\mathbb{R}^{n})} &\leq \|(f-\omega_j)\chi_{%
		\mathbb{R}^{n}\setminus\overline G}\|_{K_{p}^{\alpha,q}} +
	\|(f-\omega_j)\chi_{\overline G}\|_{K_{p}^{\alpha,q}} \\
	&= \|f\chi_{\mathbb{R}^{n}\setminus\overline G}\|_{K_{p}^{\alpha,q}} +
	\|(f-\omega_j)\chi_{\overline G}\|_{K_{p}^{\alpha,q}} \\
	&< \frac{\varepsilon}{3} + \|(f-\omega_j)\chi_{\overline
		G}\|_{K_{p}^{\alpha,q}}.
	\end{align*}
	
	Furthermore, 
	\begin{align*}
	\|(f-\omega_j)\chi_{\overline G}\|_{K_{p}^{\alpha,q}} &\leq
	\|(f-J_{\varrho}*f)\chi_{\overline G}\|_{K_{p}^{\alpha,q}} 
	+ \|(J_{\varrho}*f-\omega_j)\chi_{\overline G}\|_{K_{p}^{\alpha,q}} \\
	&< \frac{\varepsilon}{3} + \|(J_{\varrho}*f-\omega_j)\chi_{\overline
		G}\|_{K_{p}^{\alpha,q}},
	\end{align*}
	by \eqref{moll-est}.
	
	Finally, using \eqref{uniform-approx}, 
	\begin{align*}
	\|(J_{\varrho}*f-\omega_j)\chi_{\overline G}\|_{K_{p}^{\alpha,q}} &\leq 
	\frac{\varepsilon}{3F} \Bigg( \sum_{k=0}^{\infty} 2^{k\alpha q}
	|R_k\cap\overline B_v|^{q/p} \Bigg)^{1/q} \\
	&= \frac{\varepsilon}{3}.
	\end{align*}
	
	Combining the above estimates yields 
	\begin{equation*}
	\|f-\omega_j\|_{K_{p}^{\alpha,q}(\mathbb{R}^{n})} < \varepsilon.
	\end{equation*}
	Therefore, $A$ is precompact in $K_{p}^{\alpha,q}(\mathbb{R}^{n})$.
\end{proof}
In view of the proof of Theorem \ref{precompact}, together with Lemma
\ref{Key-est2} and Remark \ref{Key-est3}, we obtain the following result.

\begin{theorem}
	\label{precompact1}
	Let $\Omega\subset\mathbb{R}^{n}$ be an open set, $1\leq q<\infty$, and
	$\alpha\geq0$. Define
	\[
	r=
	\begin{cases}
	q, & \text{if }\alpha>0,\\
	1, & \text{if }\alpha=0.
	\end{cases}
	\]
	
	Suppose that $A\subset K_{1}^{\alpha,r}(\Omega)$ is bounded. Then the
	following assertions are equivalent:
	
	\begin{enumerate}
		\item[(i)] $A$ is precompact in $K_{1}^{\alpha,r}(\Omega)$;
		
		\item[(ii)] For every $\varepsilon>0$, there exist $\delta>0$ and a 
		subset $G\Subset\Omega$ such that, for every $f\in A$ and every
		$h\in\mathbb{R}^{n}$ with $|h|<\delta$,
		\begin{equation*}
		\|\tilde f(\cdot+h)-\tilde f\|_{K_{1}^{\alpha,r}(\Omega)}
		<\varepsilon,
		\end{equation*}
		and
		\begin{equation*}
		\|f\|_{K_{1}^{\alpha,r}(\Omega\setminus\overline{G})}
		<\varepsilon.
		\end{equation*}
	\end{enumerate}
\end{theorem}
The following lemma plays an important role in simplifying the arguments throughout the remainder of this section.
\begin{lemma}
	\label{compact1} Let $\Omega $ be a domain in $\mathbb{R}^{n}$ and $\Omega
	_{0}\subset \Omega $ a subdomain. Let $\alpha _{1},\alpha _{2}\in \mathbb{R}$%
	, $1\leq r<\infty $, and $1\leq q_{1}\leq q_{0}$. Assume that 
	\begin{equation}
	K_{p,m}^{\alpha _{2},r}(\Omega )\hookrightarrow K_{q_{0}}^{\alpha
		_{1},r}(\Omega _{0})  \label{embcompact}
	\end{equation}%
	continuously, and that 
	\begin{equation*}
	K_{p,m}^{\alpha _{2},r}(\Omega )\hookrightarrow K_{q_{1}}^{\alpha
		_{1},r}(\Omega _{0})\qquad \text{compactly.}
	\end{equation*}
	
	Then, for every $q_{1}\leq q<q_{0}$, the embedding 
	\begin{equation*}
	K_{p,m}^{\alpha _{2},r}(\Omega )\hookrightarrow K_{q}^{\alpha _{1},r}(\Omega
	_{0})
	\end{equation*}%
	is compact.
\end{lemma}

\begin{proof}
	Let $q_{1}\leq q<q_{0}$ and choose $0<\theta\leq1$ such that $\frac1q = \frac{1-\theta}{q_{0}} + \frac{\theta}{q_{1}}.$ Let $\{f_{j}\}$ be a bounded sequence in $K_{p,m}^{\alpha _{2},r}(\Omega )$.
	Since the embedding into $K_{q_{1}}^{\alpha _{1},r}(\Omega _{0})$ is
	compact, there exists a subsequence, still denoted by $\{f_{j}\}$, and a
	function $f\in K_{q_{1}}^{\alpha _{1},r}(\Omega _{0})$ such that 
	\begin{equation*}
	f_{j}\rightarrow f\qquad \text{in }K_{q_{1}}^{\alpha _{1},r}(\Omega _{0}).
	\end{equation*}
	
	By Lemma~\ref{interpolation2}, we have 
	\begin{equation*}
	\Vert f_{j}-f\Vert _{K_{q}^{\alpha _{1},r}(\Omega _{0})}\leq \Vert
	f_{j}-f\Vert _{K_{q_{0}}^{\alpha _{1},r}(\Omega _{0})}^{1-\theta }\Vert
	f_{j}-f\Vert _{K_{q_{1}}^{\alpha _{1},r}(\Omega _{0})}^{\theta }.
	\end{equation*}
	
	Using the continuous embedding \eqref{embcompact}, we obtain 
	\begin{equation*}
	\Vert f_{j}-f\Vert _{K_{q}^{\alpha _{1},r}(\Omega _{0})}\lesssim \Vert
	f_{j}-f\Vert _{K_{p,m}^{\alpha _{2},r}(\Omega )}^{1-\theta }\Vert
	f_{j}-f\Vert _{K_{q_{1}}^{\alpha _{1},r}(\Omega _{0})}^{\theta }.
	\end{equation*}
	
	Since $\{f_{j}\}$ is bounded in $K_{p,m}^{\alpha _{2},r}(\Omega )$, the
	first factor on the right-hand side remains bounded, while 
	\begin{equation*}
	\Vert f_{j}-f\Vert _{K_{q_{1}}^{\alpha _{1},r}(\Omega _{0})}\rightarrow 0.
	\end{equation*}%
	Therefore, 
	\begin{equation*}
	\Vert f_{j}-f\Vert _{K_{q}^{\alpha _{1},r}(\Omega _{0})}\rightarrow 0.
	\end{equation*}
	
	Hence every bounded sequence in $K_{p,m}^{\alpha _{2},r}(\Omega )$ admits a
	convergent subsequence in $K_{q}^{\alpha _{1},r}(\Omega _{0})$, which proves
	the compactness of the embedding.
\end{proof}
Now we are in a position to state the second main result of this section.
\begin{theorem}
	\label{compact2} Let $\Omega \subset \mathbb{R}^{n}$ be a domain satisfying
	the cone condition, and let $\Omega _{0}\Subset \Omega $ be a bounded
	subdomain. Let $1<p<\infty$, $1\leq r<\infty, 0\leq%
	\alpha _{1}\leq \alpha _{2}<n-\frac{n}{p}, 1\leq q<q_{0},$
	where 
	\begin{equation}
	\frac{n}{q_{0}} = \frac{n}{p}-m-\alpha _{1}+\alpha _{2} >0. 
	\end{equation}
	Assume furthermore that 
	$m>\alpha _{2}-\alpha _{1}.$
	
	Then the embedding 
	\begin{equation*}
	K_{p,m}^{\alpha _{2},r}(\Omega )\hookrightarrow K_{q}^{\alpha _{1},r}(\Omega
	_{0})
	\end{equation*}%
	is compact.
\end{theorem}

\begin{proof}
	By Lemma~\ref{compact1}, it is sufficient to prove that the embedding 
	\begin{equation}
	K_{p,m}^{\alpha _{2},r}(\Omega )\hookrightarrow K_{1}^{\alpha _{1},r}(\Omega
	_{0})   \label{embcompact2}
	\end{equation}%
	is compact. First, since $\Omega_{0}$ is bounded, the embedding \eqref{embcompact2} is continuous for any $\alpha_{1},\alpha_{2}\in\mathbb{R}$.
	
	Let $A$ be a bounded subset of $K_{p,m}^{\alpha _{2},r}(\Omega )$. We show
	that the restriction of $A$ to $\Omega _{0}$ is precompact in $K_{1}^{\alpha
		_{1},r}(\Omega _{0})$ by verifying the conditions of Theorem~\ref{precompact}.
	
	For $j\in\mathbb{N}$, define 
	\begin{equation*}
	\Omega_{j} = \Bigl\{ x\in\Omega_{0}: \mathrm{dist}(x,\partial\Omega)\geq 
	\frac{2}{j} \Bigr\}.
	\end{equation*}
	Since $\Omega_{0}\Subset\Omega$, we have 
	\begin{equation*}
	|\Omega_{0}\setminus\Omega_{j}|\to0 \qquad \text{as } j\to\infty.
	\end{equation*}
	
	Let $f\in K_{p,m}^{\alpha _{2},r}(\Omega )$ and define 
	\begin{equation*}
	\tilde{f}(x)=%
	\begin{cases}
	f(x), & x\in \Omega _{0}, \\ 
	0, & x\notin \Omega _{0}.%
	\end{cases}%
	\end{equation*}
	
	Let $k\in \mathbb{N}_{0}$, $\tilde{R}_{k}$$=$ ${B}_{0}$ if $k={0}$ and $\tilde{R}_{k}$$=$${R}_{k}$ if $k\in \mathbb{N}$ . By H\"{o}lder's inequality, 
	\begin{equation*}
	\Vert f\chi _{\tilde{R}_{k}\cap (\Omega _{0}\setminus \Omega _{j})}\Vert
	_{1}\leq |\Omega _{0}\setminus \Omega _{j}|^{1-\frac{1}{q_{0}}}\Vert f\chi _{%
		\tilde{R}_{k}\cap \Omega _{0}}\Vert _{q_{0}}.
	\end{equation*}%
	Therefore, 
	\begin{align*}
	\Vert f\Vert _{K_{1}^{\alpha _{1},r}(\Omega _{0}\setminus \Omega _{j})}&
	\leq |\Omega _{0}\setminus \Omega _{j}|^{1-\frac{1}{q_{0}}}\Vert f\Vert
	_{K_{q_{0}}^{\alpha _{1},r}(\Omega _{0})} \\
	& \lesssim |\Omega _{0}\setminus \Omega _{j}|^{1-\frac{1}{q_{0}}}\Vert
	f\Vert _{K_{p,m}^{\alpha _{2},r}(\Omega )},
	\end{align*}%
	where the last estimate follows from Theorem \ref{embeddingsfirst}.
	
	Since $q_{0}>1$ and $A$ is bounded in $K_{p,m}^{\alpha _{2},r}(\Omega )$, we
	may choose $j$ sufficiently large such that 
	\begin{equation*}
	\Vert f\Vert _{K_{1}^{\alpha _{1},r}(\Omega _{0}\setminus \Omega
		_{j})}<\varepsilon
	\end{equation*}%
	uniformly for all $f\in A$.
	
	Let $|h|$ be sufficiently small. By H\"older's inequality and Lemma~\ref{translation1}, we obtain
	\begin{align*}
	\Vert \tilde{f}(\cdot +h)\Vert _{K_{1}^{\alpha _{1},r}(\Omega _{0}\setminus
		\Omega _{j})}& \leq |\Omega _{0}\setminus \Omega _{j}|^{1-\frac{1}{q_{0}}%
	}\Vert \tilde{f}(\cdot +h)\Vert _{K_{q_{0}}^{\alpha _{1},r}(\Omega _{0})} \\
	& \lesssim |\Omega _{0}\setminus \Omega _{j}|^{1-\frac{1}{q_{0}}}\Vert 
	\tilde{f}\Vert _{K_{q_{0}}^{\alpha _{1},r}(\Omega _{0})} \\
	& \lesssim |\Omega _{0}\setminus \Omega _{j}|^{1-\frac{1}{q_{0}}}\Vert
	f\Vert _{K_{p,m}^{\alpha _{2},r}(\Omega )}.
	\end{align*}%
	Hence, 
	\begin{equation*}
	\Vert \tilde{f}(\cdot +h)-\tilde{f}\Vert _{K_{1}^{\alpha _{1},r}(\Omega
		_{0}\setminus \Omega _{j})}<\varepsilon
	\end{equation*}%
	uniformly for $f\in A$, provided $j$ is sufficiently large.
	
	Now let $f\in C^{\infty }(\Omega )\cap K_{p,m}^{\alpha _{2},r}(\Omega ),$ and assume that $|h|<\frac{1}{3j}.$
	If $x\in \Omega _{j}$ and $0\leq t\leq 1$, then $x+th\in \Omega_{2j} $. By the
	mean value formula, 
	\begin{equation*}
	|f(x+h)-f(x)|\leq |h|\int_{0}^{1}|\nabla f(x+th)|\,dt.
	\end{equation*}
	Using Minkowski's inequality and Lemma~\ref{translation1}, we obtain
	\begin{equation*}
	\Vert f(\cdot +h)-f\Vert _{K_{1}^{\alpha _{1},r}(\Omega _{j})}\lesssim
	|h|\Vert f\Vert _{K_{1,1}^{\alpha _{1},r}(\Omega _{0})}.
	\end{equation*}%
	By the continuous embedding 
	\begin{equation*}
	K_{p,m}^{\alpha _{2},r}(\Omega )\hookrightarrow K_{1,1}^{\alpha
		_{1},r}(\Omega _{0}),
	\end{equation*}%
	it follows that 
	\begin{equation}
	\Vert f(\cdot +h)-f\Vert _{K_{1}^{\alpha _{1},r}(\Omega _{j})}\lesssim
	|h|\Vert f\Vert _{K_{p,m}^{\alpha _{2},r}(\Omega )}.  \label{precompact2}
	\end{equation}
	
	Since $C^{\infty }(\Omega )\cap K_{p,m}^{\alpha _{2},r}(\Omega ) $ is dense
	in $K_{p,m}^{\alpha _{2},r}(\Omega )$, estimate \eqref{precompact2} extends
	to all $f\in K_{p,m}^{\alpha _{2},r}(\Omega )$. Therefore, for sufficiently
	small $|h|$, 
	\begin{equation*}
	\Vert \tilde{f}(\cdot +h)-\tilde{f}\Vert _{K_{1}^{\alpha _{1},r}(\Omega
		_{0})}<\varepsilon
	\end{equation*}%
	uniformly for all $f\in A$.
	
	Consequently, all conditions of Theorem~\ref{precompact} are satisfied, and
	hence $A$ is precompact in $K_{1}^{\alpha _{1},r}(\Omega _{0})$. The proof
	is complete.
\end{proof}

%\bibliography{Herz-data-base}

\begin{thebibliography}{99}
	\bibitem{AdamsFournier03} D.R. Adams, J.F. Fournier, \textit{Sobolev spaces,}
	Volume 140 of Pure and Applied Mathematics (Amsterdam), 2nd edn.
	Elsevier/Academic, Amsterdam, 2003.
	
	\bibitem{BS85} A. Baernstein II, E. T. Sawyer, \textit{Embedding and
		multiplier theorems for }$H^{p}(\mathbb{R}^{n})$, Mem. Amer. Math. Soc. 53,
	no. 318, 1985, iv+82 pp.
	
	\bibitem{Burenkov} V.I. Burenkov, \textit{Sobolev spaces on domains},
	Stuttgart: B. G. Teubner Verlagsgesellschaft mbH, 1998.
	
	\bibitem{Dr-Sobolev} D. Drihem, \textit{Herz-type Sobolev spaces on domains}%
	, Le Matematiche, 77(2) 2022, 229--263.
	
	
	\bibitem{Dr-Herz-Heat} D. Drihem, \textit{Semilinear parabolic equations in
		Herz spaces}, Appl. Anal. (2022), https://doi.
	org/10.1080/00036811.2022.2047948.
		\bibitem{Dr-Herz-Sobolev2} D. Drihem, \textit{Embedding theorems for Herz-type Sobolev spaces on domains}. Submitted.
	
	\bibitem{Fi-We08} H.G. Feichtinger, F. Weisz, \textit{Herz spaces and
		summability of Fourier transforms}, Math. Nachr, 281(3) (2008), 309--324.
	

	\bibitem{LYH} Y. Li, D. Yang, L. Huang, \textit{Real-variable theory of
		Hardy Spaces associated with generalized Herz spaces of Rafeiro and Samko},
	Lecture Notes in Mathematics 2320, Springer, Cham, 2022.
	
	
	
	\bibitem{LuYang1.95} S. Lu, D. Yang, \textit{The local versions of }$H^{p}(%
	\mathbb{R}^{n})$\textit{\ spaces at the origin}, Studia Math. 116 (1995),
	103--131.
	
	\bibitem{LuYang97} S. Lu, D. Yang, \textit{Herz-type Sobolev and Bessel
		potential spaces and their applications}, Sci in China (Ser.A). 40 (1997),
	113--129.
	
	\bibitem{LuYang2.95} S. Lu, D. Yang, \textit{The decomposition of the weighted Herz spaces on }$\mathbb{R}^{n}$\textit{\ and its applications},
	Sci.in China (Ser. A). 38 (1995), 147--158.
	
	\bibitem{Ma85} V. G. Maz'ya, \textit{Sobolev spaces}. Springer, Berlin, 1985.
	
	\bibitem{Herz68} C. Herz, \textit{Lipschitz spaces and Bernstein's theorem
		on absolutely convergent Fourier transforms}, J. Math. Mech. 18 (1968),
	283--324.
	
	\bibitem{Hernandez1998} E. Hern\'{a}ndez, D. Yang, \textit{Interpolation of
		Herz-type Hardy spaces and applications}. Math. Nachr. 42 (1998), 564--581.
	
	\bibitem{RS} H. Rafeiro, S. Samko, \textit{Herz spaces meet Morrey type
		spaces and complementary Morrey type spaces}, J. Fourier Anal. Appl. 26
	(2020), Paper No. 74, 14 pp.
	
	\bibitem{Rag09} M.A. Ragusa, \textit{Homogeneous Herz spaces and regularity
		results}, Nonlinear Anal. 71 (2009), e1909--e1914.
	
	\bibitem{Tartar} L. Tartar, \textit{An introduction to Sobolev spaces and
		interpolation spaces}, Lect. Notes Unione Mat. Ital., vol.
	3,Springer-Verlag, Berlin, Heidelberg, 2007
	
	\bibitem{Tu11} Y. Tsutsui, \textit{The Navier-Stokes equations and weak Herz
		spaces}, Advances in Differential Equations. 16 (2011), 1049--1085.
	
	\bibitem{LYH08} S. Lu, D. Yang, G. Hu, \textit{Herz type spaces and their
		applications}, Beijing: Science Press, 2008.
	
	\bibitem{ZYZ} Y. Zhao, D. Yang, Y. Zhang, \textit{Mixed-norm Herz spaces and
		their applications in related Hardy spaces}, Anal. Appl. 21 (5) (2023),
	1131--1222.
\end{thebibliography}

\end{document}